\documentclass[10pt,a4paper, reqno]{amsart}
\usepackage{amssymb, amsthm, amsmath}
\usepackage{hyperref}
\usepackage{mathrsfs}
\usepackage{color}
\usepackage{xypic}
\usepackage{enumerate}
\usepackage{setspace}

\theoremstyle{plain}
\newtheorem{theorem}{Theorem}[subsection]
\newtheorem{corollary}[theorem]{Corollary}
\newtheorem{lemma}[theorem]{Lemma}
\newtheorem{proposition}[theorem]{Proposition}
\newtheorem*{problem}{Problem}

\theoremstyle{definition}
\newtheorem{definition}{Definition}[subsection]
\newtheorem{example}[definition]{Example}

\theoremstyle{remark}
\newtheorem{remark}[definition]{Remark}

\DeclareMathOperator{\Spec}{Spec}
\DeclareMathOperator{\supp}{supp}
\DeclareMathOperator{\kk}{\textbf{k}}
\DeclareMathOperator{\KK}{K} 
\DeclareMathOperator{\rad}{rad}
\DeclareMathOperator{\Gal}{Aut} 
\DeclareMathOperator{\Hom}{Hom}
\DeclareMathOperator{\Bl}{Bl}
\DeclareMathOperator{\id}{id}

\DeclareMathOperator{\MMM}{\mathscr{M}}
\DeclareMathOperator{\NNN}{\mathscr{N}}

\DeclareMathOperator{\pr}{\textrm{pr}}

\newcommand{\OOO}{\mathcal O}

\newcommand{\mm}{\mathfrak m}
\newcommand{\nn}{\mathfrak n}
\newcommand{\pp}{\mathfrak p}

\newcommand{\QQ}{\mathbb{Q}}
\newcommand{\NN}{\mathbb{N}}
\newcommand{\ZZ}{\mathbb{Z}}
\newcommand{\PP}{\mathbb{P}}

\renewcommand{\AA}{\mathbb{A}}

\begin{document}

\title{On Maximal Subalgebras}
\author{Stefan Maubach and Immanuel Stampfli}
\address{Jacobs University Bremen gGmbH, School of Engineering and 
Science, Department of Mathematics, Campus Ring 1, 28759 Bremen, Germany}
\email{stefan.maubach@gmail.com, immanuel.e.stampfli@gmail.com}


\begin{abstract}
	Let $\kk$ be an algebraically closed field.
	We classify all maximal $\kk$-subalgebras of any 
	one-dimensional finitely generated $\kk$-domain.
	In dimension two, we classify all maximal 
	$\kk$-subalgebras of $\kk[t, t^{-1}, y]$. 
	To the authors' knowledge, this is the first such classification result 
	for an algebra of dimension $> 1$. 
	In the course of this study, we classify also all maximal
	$\kk$-subalgebras of $\kk[t, y]$ that contain a coordinate. 
	Furthermore, 
	we give examples of maximal $\kk$-subalgebras of $\kk[t, y]$
	that do not contain a coordinate.
\end{abstract}

\maketitle

\vspace{-0.075cm}

\section{Introduction}

All rings in this article are commutative an have a unity. 
A \emph{minimal ring extension} 
is a non-trivial ring extension that does not allow a proper intermediate ring. 
A good overview of minimal ring extensions can be found in 
\cite{PiPi2006About-minimal-morp}.
A first general treatment of minimal ring extensions
was done by Ferrand and Olivier in 
\cite{FeOl1970Homomorphisms-mini}. They came up with the following
important property of minimal ring extensions.

\begin{theorem}[see {\cite[Th\'eor\`eme~2.2]{FeOl1970Homomorphisms-mini}}]
	Let $A \subsetneq R$ be a minimal ring extension and let 
	$\varphi \colon \Spec(R) \to \Spec(A)$ be the induced morphism on spectra. 
	Then there exists 
	a unique maximal ideal $\mm$ of $A$ such that $\varphi$ 
	induces an isomorphism
	\[
		\Spec(R) \setminus \varphi^{-1}(\mm) 
		\stackrel{\simeq}{\longrightarrow} \Spec(A) \setminus \{ \mm \} \, .
	\]
	Moreover, the following statements are equivalent: 
	\begin{enumerate}[i)]
	 	\item The morphism $\varphi \colon \Spec(R) \to \Spec(A)$ is surjective;
		\item The ring $R$ is a finite $A$-module;
		\item We have $\mm = \mm R$.
	\end{enumerate}
\end{theorem}

Let $A \subsetneq R$ be a minimal ring extension. Then $A$ is called a 
\emph{maximal subring} of $R$.
In the case where $\Spec(R) \to \Spec(A)$ is non-surjective, 
we call $A$ an \emph{extending}\footnote{
Since \cite{FeOl1970Homomorphisms-mini} 
proves that in this case $f$ is a flat epimorphism, 
the literature calls this sometimes the ``flat epimorphism case".} maximal subring
of $R$ and otherwise, we call it a \emph{non-extending}\footnote{
The literature calls this sometimes the ``finite case".}
maximal subring.
Moreover, the unique maximal ideal $\mm$ of $A$ (from the theorem above) 
is called the \emph{crucial maximal ideal}. In this article we are interested in the description of all maximal subrings
of a given ring $R$.

Since in the non-extending case $R$ is a finite $A$-module, one might suspect, 
that this is not such a difficult case. 
Indeed, in 
Section~\ref{non-extending.sec} we provide a classification of all
non-extending maximal subrings of an arbitrary ring (up to the classification
of all maximal subfields of a given field).
In fact, Dobbs, Mullins, Picavet and Picavet-L'Hermitte gave in 
\cite{DoMuPi2005On-the-FIP-propert} 
already such a classification. The novelty of our approach is that
we explicitly construct all 
the non-extending maximal subrings of a given ring.

Thus, we are left with the extending case.
One can reduce the study to the case where $R$ is an integral domain, 
and this is explained in Section~\ref{GeneralConsiderations.sec}. 
Moreover, in
the same section, we give some general properties of maximal subrings 
that we will use often in the course of this article. 
As the classification of all extending maximal subrings of an 
arbitrary integral domain still seems to be a difficult task, we restrict ourselves
to the case where $R$ is a finitely generated domain over an 
algebraically closed field $\kk$ (of any characteristic). 
Thus our guiding problem is the following.

\begin{problem} 
	Classify all extending maximal subalgebras of a given 
	affine $\kk$-domain where $\kk$ is an algebraically closed field.
\end{problem}

If the affine $\kk$-domain is one-dimensional, then we are able to describe all 
extending maximal $\kk$-subalgebras. 
This description is provided in Section~\ref{OneDimensionalCase.sec}. 
Let us give a simple example.

\begin{example} If $R = \kk[t, t^{-1}]$,
then the only extending maximal $\kk$-sub\-algebras are $\kk[t]$ and 
$\kk[t^{-1}]$. In fact, $\PP_{\kk}^1$ is a smooth projective closure of 
$\AA_{\kk}^\ast$. If we identify $\AA_{\kk}^\ast$ with the image under 
the open immersion
\[
	\AA_{\kk}^\ast \longrightarrow \PP^1_{\kk} \, , \quad
	t \longmapsto (t: 1) \, ,
\] 
then $\kk[t]$ is the subring of functions on 
$\AA_{\kk}^\ast$ that are defined at
$(0:1) \in \PP^1_{\kk}$ and $\kk[t^{-1}]$ is the subring of those functions on 
$\AA_{\kk}^\ast$
that are defined at $(1:0)$.
\end{example}

This example is an instance of the general description.

\begin{theorem}[see Theorem~\ref{thm:OneDimensional}]
	Let $R$ be a finitely generated one-dimensional $\kk$-domain. 
	Take a projective closure $\overline{X}$ of the affine curve $X = \Spec(R)$
	such that $\overline{X}$ is non-singular at every point of
	$\overline{X} \setminus X$. If $\overline{X} \setminus X$ contains just 
	a single point, then $R$ has no extending maximal $\kk$-subalgebra. 
	Otherwise, for any point 
	$p \in \overline{X} \setminus X$,
	\[	
		\{ \, f \in R \ | \ \textrm{$f$ is defined at $p$} \, \}
	\]
	is an extending maximal $\kk$-subalgebra of $R$ and every extending 
	maximal $\kk$-subalgebra of $R$ is of this form.
\end{theorem}

In dimension two, the most natural algebra to study is the polynomial
algebra in two variables $\kk[t, y]$.
Using the classification of 
extending maximal subalgebras of a one-dimensional affine $\kk$-domain,
we give in Section~\ref{sec.ExamplesContainNotACoordinate} 
plenty examples of extending maximal subalgebras of $\kk[t, y]$
that do not contain a coordinate of $\kk[t, y]$, 
i.e.~they do not contain a polynomial in $\kk[t, y]$ which 
is the component of an automorphism of $\AA^2_{\kk}$. 
These examples indicate that it is
difficult to classify \emph{all} extending maximal subalgebras of $\kk[t, y]$. 
Therefore, we impose more structure in the problem. 
Namely, we search for all extending maximal 
subalgebras of $\kk[t, y]$ that contain a coordinate of 
$\kk[t, y]$.

Another natural $2$-dimensional affine $\kk$-domain beside 
the polynomial algebra $\kk[t, y]$ is the localization of it in $t$, 
i.e. the $2$-dimensional domain $\kk[t, t^{-1}, y]$. This algebra is directly
related to our former problem, as it is isomorphic to the localization of
$\kk[t, y]$ in any coordinate of $\kk[t, y]$. In fact, in this article 
we classify all extending maximal subalgebras of $\kk[t, t^{-1}, y]$
and get in the course of this classification all extending maximal subalgebras
of $\kk[t, y]$ that contain a coordinate. This is the bulk of this article.
To the authors' knowledge, this is the first such classification result for 
an algebra of dimension $> 1$.

Let us give an instructive example, before we give more details on our results.
\begin{example}[see Lemma~\ref{lem:maxandcond}]
\label{exa:motivatingexample}
Let $\kk$ be an algebraically closed field and let $R = \kk[t, t^{-1}, y]$.
The ring
\[
	A=\kk[t] + y\kk[t, t^{-1}, y] =  \kk[t, y, y/t, y/t^2, y/t^3, \, \ldots]
\] 
is an extending maximal subalgebra of $\kk[t, t^{-1}, y]$.
The crucial maximal ideal of $A$ is given by 
\[
	\mm = (t, y, y/t, y/t^2, \ldots) \, .
\]
Thus $A \subseteq \kk[t, t^{-1}, y]$ induces an open immersion
$\AA^\ast_{\kk} \times \AA_{\kk}^1 \to \Spec(A)$ and the complement of the image
is just $\{ \mm \}$. Moreover, the morphism 
$\Spec(A) \to \AA^2_{\kk}$ induced by $\kk[t, y] \subseteq A$, 
sends the crucial maximal ideal $\mm$ to the origin $(0,0)$. 
So in some sense we ``added" to $\AA_{\kk}^{\ast} \times \AA_{\kk}^1$ the point 
$(0, 0) \in \{ 0 \} \times \AA_{\kk}^1$.

Another description of the affine scheme $\Spec(A)$ is the following: It
is the inverse limit of
	$\ldots {\longrightarrow} 
	\AA_{\kk}^2 \stackrel{\varphi}{\longrightarrow} 
	\AA_{\kk}^2 \stackrel{\varphi}{\longrightarrow} 
	\AA_{\kk}^2$
inside the category of affine schemes, where $\varphi(t, x) = (t, t x)$. 

A little more general, for any $\alpha \in \kk[t]$, the ring 
$\kk[t] + (y-\alpha)\kk[t, t^{-1}, y]$
is also an extending maximal subalgebra of $\kk[t, t^{-1}, y]$.
\end{example}

Towards the classification of all extending maximal subalgebras
of $\kk[t, t^{-1}, y]$, we describe in 
Section~\ref{ClassContainKty.sec} all extending maximal 
$\kk$-subalgebras of $\kk[t, t^{-1}, y]$ 
that contain $\kk[t, y]$. To formulate our results we introduce some
notation. Let $\kk[[t^\QQ]]$ be the Hahn field over $\kk$
with rational exponents, i.e. the field of formal power series
\[
	\alpha = \sum_{s \in \QQ} a_s t^s \quad 
	\textrm{such that 
	$\supp(\alpha) = \{ \, s \in \QQ \ | \ a_s \neq 0 \, \}$ is well ordered.} 
\]
Moreover, we denote by $\kk[[t^\QQ]]^+$ the subring 
of elements $\alpha \in \kk[[t^\QQ]]$ that satisfy $\supp(\alpha)Ê\subseteq [0, \infty)$.
By extending the scalars $\kk[t, t^{-1}]$ to the Hahn field $\kk[[t^\QQ]]$
one has a simple classification:

\begin{theorem}[see Corollary~\ref{cor:HahnClass} and 
	Remark~\ref{rem:HahnClass}]
	We have a bijection
	\[
		\kk[[t^\QQ]]^+ \longrightarrow
		\left\{ 
			\begin{array}{c}
				\textrm{extending maximal } \\
				\textrm{$\kk$-subalgebras of $\kk[[t^\QQ]][y]$} \\
				\textrm{that contain $\kk[[t^\QQ]]^+[y]$}
			\end{array} 
		\right\}
		\, , \quad
		\alpha \longmapsto \kk[[t^\QQ]]^+ + (y-\alpha) \kk[[t^\QQ]][y]
	\]
\end{theorem}

With the aid of this theorem, we are able to classify all
extending maximal $\kk$-subalgebras of $\kk[t, t^{-1}, y]$ that contain $\kk[t, y]$.

\begin{theorem}[see Theorem~\ref{thm:Uniqueness}]
	\label{MainthmC}
	Let $\mathscr{S}$ be the set of
	$\alpha \in \kk[[t^\QQ]]^+$ such that $\supp(\alpha)$
	is contained in a strictly increasing sequence of $\QQ$.
	Then we have a surjection 
	\[
		\mathscr{S} \longrightarrow 
		\left\{ 
		\begin{array}{c}
		 	\textrm{extending maximal} \\ 
			\textrm{$\kk$-subalgebras of $\kk[t, t^{-1}, y]$} \\
			\textrm{that contain $\kk[t, y]$}
		\end{array}
		\right\} \, ,
		\quad
		\alpha \longmapsto A_\alpha \cap \kk[t, t^{-1}, y]
	\]
	where
	\[
		A_{\alpha} = \kk[[t^\QQ]]^+ + (y-\alpha) \kk[[t^\QQ]][y] \, .
	\]
	Moreover, two elements of $\mathscr{S}$ are sent to the same 
	$\kk$-subalgebra,
	if and only if they lie in the same orbit under the natural action
	of $\Hom(\QQ/\ZZ, \kk^\ast)$ on $\mathscr{S}$.
\end{theorem}

In Section~\ref{classOfk[tt-1y].sec} we start with the description of \emph{all}
maximal $\kk$-subalgebras of $\kk[t, t^{-1}, y]$. Our main result of that section is
the following.

\begin{theorem}[see Proposition~\ref{prop:classOfk[tt-1y]}]
	\label{MainthmD}
	Let $A \subseteq \kk[t, t^{-1}, y]$ be an extending 
	maximal $\kk$-subalgebra. Then, exactly one of the following cases occur:
	\begin{enumerate}[i)]
	\item There exists an automorphism $\sigma$ of $\kk[t, t^{-1}, y]$
		such that $\sigma(A)$ contains $\kk[t, y]$;
	\item $A$ contains $\kk[t, t^{-1}]$.
	\end{enumerate}
\end{theorem}

The maximal $\kk$-subalgebras of case i) are described by 
Theorem~\ref{MainthmC}. Thus, we are left with the description of the
extending maximal $\kk$-subalgebras of $\kk[t, t^{-1}, y]$ that contain 
$\kk[t, t^{-1}]$. This will be done in Section~\ref{ClassWithCoord.sec}.
In order to state our result let us introduce some
notation. Let $\MMM$ be the set of extending maximal $\kk$-subalgebras of 
$\kk[t, y]$ that contain $\kk[t]$. Moreover,      
let $\NNN$ be the set of extending maximal $\kk$-subalgebras $A$ of 
$\kk[t, y, y^{-1}]$ that contain $\kk[t, y^{-1}]$ 
and such that
\[
	A  \longrightarrow \kk[t, y, y^{-1}] / (t-\lambda)
\] 
is surjective,
where $\lambda$ is the unique element in $\kk$ such that the crucial
maximal ideal of $A$ contains $t-\lambda$ (this $\lambda$ exists
by Remark~\ref{rem:TothesetN}). 
The set $\NNN$ is described by 
Theorem~\ref{MainthmC}. Then, the maximal $\kk$-subalgebras of case ii) in 
Theorem~\ref{MainthmD} are described by the following result.

\begin{theorem}[see Theorem~\ref{thm:MandN} 
and Proposition~\ref{prop:ExtMaxContainingtandtinv}]
	\label{MainthmE}
	With the definitions of $\MMM$ and of $\NNN$ from above, 
	we have bijections $\Theta$ and $\Phi$
	\[
		\NNN \stackrel{\Theta}{\longrightarrow}
		\MMM
		\supseteq
		\left\{
			\begin{array}{c}
				\textrm{$B$ in $\MMM$ s.t. the crucial} \\
				\textrm{maximal ideal of $B$} \\ 
				\textrm{does not contain $t$}
			\end{array}
		\right\}
		\stackrel{\Phi}{\longleftarrow}
				\left\{
			\begin{array}{c}
				\textrm{extending maximal $\kk$-} \\
				\textrm{subalgebras of $\kk[t, t^{-1}, y]$} \\
				\textrm{that contain $\kk[t, t^{-1}]$}
			\end{array}
		\right\}
	\]
	given by $\Theta(A) = A \cap \kk[t, y]$
	and $\Phi(A') = A' \cap \kk[t, y]$. 
\end{theorem}

In particular, with the aid of the bijection $\Theta \colon \NNN \to \MMM$ in 
Theorem~\ref{MainthmE} we get a description of the
extending maximal $\kk$-subalgebras of $\kk[t, y]$ that contain a coordinate
of $\kk[t, y]$.

%
%

\section{Classification of the non-extending maximal subrings}
\label{non-extending.sec}

Let $R$ be any ring and denote by $X = \Spec(R)$ the
corresponding affine scheme. In the sequel, we describe 
three procedures to construct a non-extending maximal subring of $R$.
\medskip

\textbf{a) Glueing two closed points transversally.}
Choose two different closed points $x_1, x_2 \in X$ such that their
residue fields $\kappa(x_1)$, $\kappa(x_2)$ are isomorphic and choose
some isomorphism $\sigma \colon \kappa(x_1) \to \kappa(x_2)$. Let
\[
	R_{x_1, x_2} = 
	\{ \, f \in R \ | \ \sigma(f(x_1)) = f(x_2) \, \}
\]
(note that $R_{x_1, x_2}$ depends on $\sigma$).
Then $R_{x_1, x_2}$ is a non-extending 
maximal subring of $R$ with crucial maximal ideal
\[
	\mm_{x_1, x_2} = \{ \, f \in R \ | \ \sigma(f(x_1)) = f(x_2) = 0 \, \} \, ,
\]
the homomorphisms on residue fields
$R_{x_1, x_2} / \mm_{x_1, x_2} \to  \kappa(x_i)$ are isomorphisms
and the fiber of $X \to \Spec(R_{x_1, x_2})$ over $\mm_{x_1, x_2}$ 
contains only the points $x_1$ and $x_2$. Moreover, the natural linear map 
on tangent spaces
\[
	T_{x_1} X \oplus T_{x_2} X \longrightarrow 
	T_{\mm_{x_1, x_2}} \Spec(R_{x_1, x_2})
\]
is an isomorphism.

\begin{proof}
	For $i = 1, 2$, let $\mm_i \subseteq R$ be the maximal ideal corresponding 
	to $x_i$. The injective homomorphism
	\[
		R_{x_1, x_2} / \mm_{x_1, x_2} \longrightarrow R / \mm_i
	\]
	is surjective. Indeed, let $f \in R$. By symmetry, we can assume that
	$i = 1$. Then there exists $h \in \mm_1$
	such that $h(x_2) = \sigma(f(x_1))-f(x_2)$. Thus
	$f + h \in R_{x_1, x_2}$, which proves the surjectivity. Let
	$\kappa = R_{x_1, x_2} / \mm_{x_1, x_2}$.
	Since $\mm_{x_1, x_2} = \mm_1 \cap \mm_2$, the homomorphism
	\[
		R_{x_1, x_2} / \mm_{x_1, x_2} \subseteq R / \mm_{x_1, x_2} = 
		R / \mm_1 \cap \mm_2 = R / \mm_1 \times R / \mm_2
	\]
	identifies with the diagonal homomorphism $\kappa \to \kappa \times \kappa$.
	Since $\mm_{x_1, x_2} = \mm_{x_1, x_2} R$, it follows that $R_{x_1, x_2}$
	is a maximal subring of $R$, see Lemma~\ref{lem:maxandcond}.
	Moreover, the fiber of
	$X \to \Spec(R_{x_1, x_2})$ over $\mm_{x_1, x_2}$ consists of $x_1$
	and $x_2$. For the last statement we prove that
	the $\kappa$-linear map on cotangent spaces
	\begin{equation}
		\label{injectivity.eq}
		\mm_{x_1, x_2} / \mm_{x_1, x_2}^2
		\longrightarrow  \mm_1 / \mm_1^2 \oplus \mm_2 / \mm_2^2
	\end{equation}
	is an isomorphism.
	Let $f_1 \in \mm_1$ such that $f_1(x_2) = 1$.
	The ideals $\mm_1^2$ and $\mm_2^2$ are coprime, since
	$\mm_1$ and $\mm_2$ are coprime,
	and thus we get $\mm_1^2 \cap \mm_2^2 = \mm_1^2 \cdot \mm_2^2
	\subseteq \mm_{x_1, x_2}^2$. 
	This proves the injectivity of \eqref{injectivity.eq}. Let
	$h \in \mm_2$. Then
	\[
		f_1^2 h \in \mm_{x_1, x_2} \, , \quad f_1^2 h - h \in \mm_2^2
		\quad \textrm{and} \quad f_1^2 h \in \mm_1^2 \, ,
	\]
	which proves that $\{ 0 \} \oplus \mm_2 / \mm_2^2$ 
	lies in the image of \eqref{injectivity.eq}. By symmetry we
	get the surjectivity of \eqref{injectivity.eq}.
\end{proof}

\medskip
\textbf{b) Deleting a tangent direction at a closed point.} 
Choose a closed
point $x \in X$ and a derivation 
$\delta \colon \OOO_{X, x} \to \kappa(x)$ 
that induces a non-zero tangent vector
in $T_x X$. Let
\[
	R_{x, \delta} = \{ \, f \in R \ | \ \delta(f) = 0 \, \} \, .
\]
Then $R_{x, \delta}$ is a non-extending maximal subring of $R$ with crucial maximal ideal
\[
	\mm_{x, \delta} = \{ \, f \in R \ | \ \delta(f) = 0 \, , \, f(x) = 0 \, \} \, .
\]
The morphism $X \to \Spec(R_{x, \delta})$ is bijective, 
maps $x$ on $\mm_{x, \delta}$ and induces an isomorphism
on residue fields $R_{x, \delta} / \mm_{x, \delta} \to \kappa(x)$. Moreover there
is an induced exact sequence
\[
	0 \longrightarrow \kappa(x) v \longrightarrow T_x X \longrightarrow
	T_{\mm_{x, \delta}} \Spec(R_{x, \delta})
\]
where $vÊ\in T_x X$ denotes the non-zero tangent vector induced by $\delta$.

\begin{proof}
	Let $\mm \subseteq R$ be the maximal ideal corresponding to $x$ in $X$.
	Since $\delta$ induces a non-zero 
	$\kappa(x)$-linear map $\mm / \mm^2 \to \kappa(x)$,
	there exists $e \in \mm$ such that $\delta(e) = 1$.
	The injective homomorphism
	\[
		R_{x, \delta} / \mm_{x, \delta} \longrightarrow \kappa(x)
	\]
	is surjective. Indeed, if $f \in R$,
	then there exists $r \in R$ such that $\delta(f)-r(x)= 0$ inside $\kappa(x)$. 
	Hence,
	$f - er \in R_{x, \delta}$ and $er \in \mm$, which proves the surjectivity. 
	
	Let $\kappa = R_{x, \delta} / \mm_{x, \delta}$.
	Since $\mm_{x, \delta}$ is an ideal of $R$, we get a $\kappa$-algebra 
	isomorphism
	\[
		\kappa[\varepsilon] / (\varepsilon^2) 
		\longrightarrow R / \mm_{x, \delta} \, , \quad 
		\varepsilon \longmapsto e
	\]
	and the map $R_{x, \delta} / \mm_{x, \delta} \to
	R / \mm_{x, \delta}$ identifies with the $\kappa$-linear map
	$\kappa \to \kappa[\varepsilon] / (\varepsilon^2)$. 
	By Lemma~\ref{lem:maxandcond}, it follows 
	that $R_{x, \delta}$ is a maximal subring of $R$. Clearly, the induced
	map $X \to \Spec(R)$ is bijective and maps $x$ to $\mm_{x, \delta}$.
	The last statement follows from the exact sequence of $\kappa$-vector
	spaces
	\[
		\mm_{x, \delta} / \mm_{x, \delta}^2 \longrightarrow \mm / \mm^2 
		\stackrel{\delta}{\longrightarrow} 
		\kappa \longrightarrow 
		0 \, .
	\]
\end{proof}

\begin{remark} $ $
	\begin{enumerate}[i)]
	\item
	There are affine schemes $X$ such that for some closed point 
	$x \in X$ the tangent space $T_x X \neq 0$, even though there 
	exists no non-zero derivation
	$\OOO_{X, x} \to \kappa(x)$. Take for example $X = \Spec(\ZZ)$
	and $x = p \ZZ$ where $p$ is some prime number.
	
	\item There are affine schemes $X$ such that for some
	closed point $x \in X$ there exists a non-zero derivation
	$\delta \colon \OOO_{X, x} \to \kappa(x)$ that induces the zero 
	vector in $T_x X$. Take for example any affine scheme $X$
	and take any closed point $x \in X$ such that there exists a
	non-zero derivation $\delta_0 \colon \kappa(x) \to \kappa(x)$.
	Then  
	\[
		\OOO_{X, x} \longrightarrow \kappa(x) 
		\stackrel{\delta_0}{\longrightarrow} \kappa(x)
	\]
	is a non-zero derivation that induces the zero vector in $T_x X$.
	\end{enumerate}
\end{remark}


\medskip
\textbf{c) Shrinking the residue field at a closed point.}
Choose a closed point $x \in X$ and choose a maximal subfield
$k$ of the residue field $\kappa(x)$, i.e. a subfield $k \subseteq \kappa(x)$
such that there exists no proper intermediate field between $k$ and 
$\kappa(x)$. Let
\[
	R_{x, k} = \{ \, f \in R \ | \ f(x) \in k \, \} \, .
\]
Then $R_{x, k}$ is a non-extending maximal subring of $R$ with crucial
maximal ideal
\[
	\mm_{x, k} = \{ \, f \in R \ | \ f(x) = 0  \, \}
\]
and the residue field $R_{x, k} / \mm_{x, k}$ is $k$. Moreover, the morphism
$X \to \Spec(R_{x, k})$ is bijective and maps $x$ on $\mm_{x, k}$.

\begin{proof}
	Note, that a maximal subfield of a field is automatically a maximal subring
	of that field. Thus, the maximality of $R_{x, v}$ in $R$ follows from the 
	fact that
	\[
		k = R_{x, v} / \mm_{x, v} \longrightarrow R / \mm_{x, v} = \kappa(x)
	\]
	is a maximal subfield, see Lemma~\ref{lem:maxandcond}.
	The other statements are clear.
\end{proof}

The next result shows, that every non-extending maximal subring arises
by one of the three constructions above. In fact, a version of this
result can be found in \cite[Corollary II.2]{DoMuPi2005On-the-FIP-propert}.
However, for the sake of completeness and the shortness of the argument, 
we provide a proof.
The main ingredient will be an easy, but very important Lemma of
Ferrand and Olivier~\cite{FeOl1970Homomorphisms-mini}.

\begin{proposition}
	\label{prop:Non-extending}
	If $A \subseteq R$ is a non-extending maximal subring, then 
	it is one of the maximal subrings constructed in a), b) or c).
\end{proposition}

\begin{proof}
	By assumption, the map 
	$\Spec(R) \to \Spec(A)$ is surjective and there exists a unique maximal ideal
	$\mm \subseteq A$ such that $\mm = \mm R$. By 
	Lemma~\ref{lem:maxandcond}, the field $K = A/ \mm$ is a maximal subring of
	$R / \mm$. By \cite[Lemme~1.2]{FeOl1970Homomorphisms-mini} 
	one of the following possibilities occur:
	\begin{enumerate}[i)]
		\item The map $K \subseteq R / \mm$ identifies with the 
			 diagonal map $K \to K \times K$.
		\item The map $K \subseteq R / \mm$ identifies with the 
		 	$K$-homomorphism $K \to 
			K[\varepsilon] / (\varepsilon^2)$.
		\item $R / \mm$ is a field.
	\end{enumerate}
	If we are in case i), then $A = R_{x_1, x_2}$ where
	$x_1, x_2 \in \Spec(R)$ correspond to the two maximal ideals
	$\{ 0 \} \times K$, $K \times \{ 0 \}$ of
	$R / \mm \simeq K \times K$ and $\sigma$ is given by
	\[
		\kappa(x_1) \stackrel{\simeq}{\longleftarrow} A / \mm
		\stackrel{\simeq}{\longrightarrow} \kappa(x_2) \, .
	\] 
	If we are in case ii), then
	$A = R_{x, \delta}$ where $x \in \Spec(R)$ corresponds to the 
	maximal ideal $(\varepsilon)$ of 
	$R / \mm \simeq K[\varepsilon] / (\varepsilon^2)$
	and $\delta$ is given by
	\[
		\delta \colon R \longrightarrow R / \mm \simeq 
		K[\varepsilon] / (\varepsilon^2) 
		\stackrel{\delta'}{\longrightarrow} K
	\]
	where $\delta'$ is the $K$-derivation that maps $\varepsilon$ to $1$.
	If we are in case iii), then $A = R_{x, k}$ where $x$ corresponds to
	the maximal ideal $\mm$ of $R$ and $k$ is the subfield 
	$K$ of $R / \mm$. This finishes the proof.
\end{proof}

To the authors' knowledge there is no complete description of the 
maximal subfields of a given field. Partial results in this direction
can be found in \cite[Proposition~2.2]{PiPi2006About-minimal-morp}.
However, if we restrict ourselves to the case, where $R$ is a finitely
generated algebra over an algebraically closed field and if we consider only 
subalgebras, then we can exclude the construction c). More precisely, we get
the following result.

\begin{corollary}
	Let $\kk$ be an algebraically closed field and let $R$ be a finitely
	generated $\kk$-algebra. Denote $X = \Spec(R)$. Then
	\begin{enumerate}[i)]
		\item
		$\{ \, f \in R \ | \ f(x_1) = f(x_2) \ \textrm{inside $\kk$} \,Ê\}$ \quad (where
		$x_1 \neq x_2 \in X$ are closed points)
		\item
		$\{ \, f \in R \ | \ D_v(f) = 0 \, \}$ \quad (where $0 \neq v \in T_x X$
		and $x \in X$ is a closed point)
	\end{enumerate}
	are non-extending maximal $\kk$-subalgebras of $R$.
	Moreover, every non-extending maximal $\kk$-subalgebra of $R$ is
	one of the above. 
\end{corollary}

\begin{proof}
	For closed points $x_1 \neq x_2 \in X$, it follows by the construction in a)
	that the $\kk$-subalgebra in i) is equal to $R_{x_1, x_2}$
	(we define $\sigma \colon \kk \to \kk$ as the identity). For a closed point
	$x \in X$ and for a non-zero tangent vector $0 \neq v \in T_x X$, it 
	follows by the construction in b) that the $\kk$-subalgebra in ii) is equal to 
	$R_{x, \delta}$ where we define $\delta$ as the $\kk$-derivation
	\[
		D_v \colon \OOO_{X, x} = \mm_x \oplus \kk \cdot 1 \longrightarrow 
		\mm_x \longrightarrow \mm_x / \mm_x^2 
		\stackrel{v}{\longrightarrow} \kk
	\]
	and where $\mm_x \subseteq \OOO_{X, x}$ 
	denotes the unique maximal ideal.
	
	Conversely, let $A \subseteq R$ be a non-extending 
	maximal $\kk$-subalgebra. By Proposition~\ref{prop:Non-extending}, 
	$A$ is one of the maximal
	subrings of $R$ constructed in a), b) or c).
	Since $R$ is a finitely generated $\kk$-algebra
	and since $\kk$ is algebraically closed, it follows that the residue field
	of every closed point of $X$ is $\kk$. Since $R$ is a
	finitely generated $A$-module, it follows that $X \to \Spec(A)$ is surjective,
	and therefore the residue field of every closed point of $\Spec(A)$ is $\kk$.
	Thus $A$ cannot be one of the maximal subrings constructed in c).
	We distinguish two cases.
	\begin{itemize}
		\item $A$ is the maximal subring constructed in a). Then there
		exist two different closed points $x_1 \neq x_2$ in $X$
		and an isomorphism 
		$\sigma \colon \kappa(x_1) \to \kappa(x_2)$ such that
		\[
			A = \{ \, f \in R \ | \ \sigma(f(x_1)) = f(x_2) \,Ê\} \, .
		\] 
		Since $\kk \subseteq A$, $\sigma$ commutes with
		the canonical isomorphisms $\kk \simeq \kappa(x_i)$.
		Thus by identifying $\kappa(x_i)$ with $\kk$, the isomorphism
		$\sigma$ is the identity.
		
		\item $A$ is the maximal subring constructed in b). Then there exist
		a closed point $x \in X$ and a derivation 
		$\delta \colon \OOO_{X, x} \to \kk$ (that induces a
		non-zero tangent vector at $x$) such that
		\[
			A = \{ \, f \in R \ | \ \delta(f) = 0 \,Ê\} \, .
		\]
		Since $\kk \subseteq A$, it follows that $\delta$ is a $\kk$-derivation.
		There exists a $\kk$-linear isomorphism from $T_x X$ 
		to the vector space of $\kk$-derivations $\OOO_{X, x} \to \kk$, given
		by $v \mapsto D_v$. Hence, 
		there exists $0 \neq v \in T_x X$ such that $\delta = D_v$.
	\end{itemize}
	This finishes the proof.
\end{proof}

%

\section{Some general considerations about maximal subrings}
\label{GeneralConsiderations.sec}

\subsection{Reduction to integral domains}

The aim of this subsection is to show that one can reduce the classification
of the extending maximal subrings of a ring $R$ to the case, 
where $R$ is an integral domain.

Let $\pp$ be a minimal prime ideal of $R$ and denote by 
$\pi \colon R \to R / \pp$ the canonical projection. 
If $A$ is an extending maximal subring of $R$, such that the crucial maximal ideal
contains $A \cap \pp$, then one can easily see, that $\pi(A)$ is an extending
maximal subring of $R / \pp$. On the other hand, if $B$ is an extending maximal
subring of $R / \pp$, then one can easily see $\pi^{-1}(B)$ is an extending
maximal subring of $R$ and its crucial maximal ideal contains $\pi^{-1}(B) \cap \pp$.
Thus we established a bijective correspondence:
\[
	\left\{
	\begin{array}{c}
	    \textrm{extending maximal subrings $A \subsetneq R$} \\
	    \textrm{such that the crucial maximal} \\
	    \textrm{ideal contains $A \cap \pp$}
	\end{array}
	\right\}
	\stackrel{1:1}{\longleftrightarrow}
	\left\{
	\begin{array}{c}
	    \textrm{extending maximal} \\ 
	    \textrm{subrings of $R / \pp$}
	\end{array}
	\right\} \, .
\]
Note that for every extending maximal subring $A$ of $R$ there exists a minimal
prime ideal $\pp$ of $R$ such that the crucial maximal ideal contains $A \cap \pp$.
Thus we are reduced to the case, where $R$ is an integral domain.

\subsection{Some properties of maximal subrings}
\label{Generalfacts.subsec}
In this subsection we gather some general properties of 
maximal subrings, that we will
constantly use in the course of this article.

The first lemma says that maximal subrings behave well under localization.

\begin{lemma}[see {\cite[Lemme~1.3]{FeOl1970Homomorphisms-mini}}]
	\label{lem:MaxSubringAndLocalization}
	Let $A \subseteq R$ be a maximal subring and let $S$ be 
	a multiplicatively closed subset of $A$. 
	Then the localization $A_S$ is either a maximal subring
	of  the localization $R_S$ or $A_S = R_S$.
\end{lemma}

The second lemma gives us the possibility for certain cases 
to reduce to quotient rings, while searching
for maximal subrings. It is a direct consequence of 
\cite[Lemme~1.4]{FeOl1970Homomorphisms-mini}.

\begin{lemma}
	\label{lem:maxandcond}
	Let $A \subseteq R$ be a ring extension and let $I \subseteq A$ be an
	ideal such that $I = I R$. Then 
	$A$ is a maximal subring of $R$
	if and only if $A / I$ is a maximal subring of $R / I$.
\end{lemma}

In particular, for every ring extension $A \subseteq R$, 
the conductor ideal
\[
	I =  \{ \, a \in A \ | \ aR \subseteq A \, \}
\]
satisfied $I = I R$. 
Note that every ideal $J$ of $A$ with $J = J R$ is contained in 
the conductor ideal $I$.

\begin{lemma}[see {\cite[Lemme~3.2]{FeOl1970Homomorphisms-mini}}]
	Let $A \subsetneq R$ be an extending maximal subring. Then the
	conductor ideal of $A$ in $R$ is a prime ideal of $R$.
\end{lemma}

Samuel introduced in \cite{Sa1957La-notion-de-place} the $P_2$-property for
ring extensions. This property will be 
crucial for our classification result.

\begin{definition}
	Let $A \subseteq R$ be a subring. We say that $A$ 
	satisfies the \emph{property $P_2$ in $R$}, if
	for all $r, q \in R$ with $r q \in A$ we have either 
	$r \in A$ or $q \in A$.
\end{definition}

\begin{lemma}[see {\cite[Proposition~3.1]{FeOl1970Homomorphisms-mini}}]
	\label{lem:PropertyP_2ForExtMax}
	Let $A \subsetneq R$ be an extending maximal subring.
	Then $A$ satisfies the property $P_2$ in $R$.
\end{lemma}

The next lemma shows, that
the extending maximal subrings of a field have a well known characterization.
It is a direct consequence of 
\cite[Proposition~3.3]{FeOl1970Homomorphisms-mini}.

\begin{lemma}
	\label{lem:ExtendingMaxOfField}
	Let $K$ be a field and let $R \subsetneq K$ be a subring. Then,
	$R$ is an extending maximal subring of $K$ if and only if $R$ is a
	one-dimensional valuation ring of $K$.
\end{lemma} 

Let us state and prove the following rather technical lemma for future use.

\begin{lemma}
	\label{lem:ResidueField}
	Let $C$ be a Noetherian domain such that the quotient field
	$Q(C)$ is not a finitely
	generated $C$-algebra. Let $A \subsetneq C[y]$ be an extending
	maximal subring that contains $C$ and denote by $\mm$ the 
	crucial maximal ideal of $A$. Then $\mm \cap C \neq 0$.
\end{lemma}

\begin{proof}
	Assume that $\mm \cap C = 0$. Then we have the following commutative
	diagram
	\[
		\xymatrix{
			C \ar[r] \ar[d] & A \ar@{->>}[d]^-{\pi} \ar[r] & C[y] \\
			Q(C) \ar[r] & A / \mm \, .
		}
	\]
	As $A$ is a maximal subring of $C[y]$, 
	we have $A \not\subseteq C$ and thus there exists
	$f \in A$ with $\deg_y(f) > 0$. Let $f = f_n y^n + \ldots + f_1 y + f_0$
	where $f_i \in C$, $f_n \neq 0$. We have
	\[
		y (f_n y^{n-1} + \ldots + f_1) = f - f_0 \in A \, .
	\]
	Since $A$ satisfies the property $P_2$ in $C[y]$ and since
	$y \not\in A$ we get $f_n y^{n-1} + \ldots + f_1 \in A$.
	Proceeding in this way it follows that there exists $0 \neq c \in C$
	such that $cy \in A$.
	Let us define the $C$-algebra homomorphism $\sigma$ by
	\[
		\sigma \colon C[y] \longrightarrow A / \mm \, , \quad
		y \mapsto \frac{\pi(cy)}{\pi(c)} \, .
	\]
	We claim that $\sigma$ and $\pi$ coincide on $A$.
	We proceed by induction on the $y$-degree of the elements in $A$. 
	By definition, $\sigma$ and 
	$\pi$ coincide on $C$, i.e. they coincide on 
	the elements of $y$-degree equal to zero.
	Let $g = g_n y^n + \ldots + g_1 y + g_0 \in A$ and assume that 
	$g_n \neq 0$, $n > 0$.
	As before, we get $y(g_n y^{n-1} + \ldots + g_1) \in A$ and
	$g_n y^{n-1} + \ldots + g_1 \in A$.
	Thus we have
	\begin{align*}
		\pi(g) &= \pi(y(g_n y^{n-1} + \ldots + g_1)) + \pi(g_0) \\
			 &= \frac{\pi(cy(g_n y^{n-1} + \ldots + g_1))}{\pi(c)} + \pi(g_0) \\
			 &= \frac{\pi(cy)}{\pi(c)} \pi(g_n y^{n-1} + \ldots + g_1) + \pi(g_0) \\
			 &= \sigma(y) \sigma(g_n y^{n-1} + \ldots + g_1) + \sigma(g_0) \\
			 & = \sigma(g) \, ,
	\end{align*}
	where we used in the second last equality the induction hypothesis. 
	This proves the claim. Since $\pi$ is surjective, $\sigma$ is surjective
	too. Hence there exists $a \in A / \mm$ which is algebraic over $Q(C)$
	such that $A/\mm$ is generated by $a$ as a $C$-algebra.
	Let $h_0 + h_1 x + \ldots + h_m x^m + x^{m+1}$ be the 
	minimal polynomial of $a$ over $Q(C)$ and let
	\[
		C_0 = C[h_0, \ldots, h_m] \subseteq Q(C) \, .
	\]
	As $C$ is Noetherian, $C_0$ is Noetherian. Moreover,
	$A / \mm$ is generated by $1, a, \ldots, a^m$ as a $C_0$-module.
	Hence, $Q(C)$ is a finitely generated $C_0$-module.
	Thus $Q(C)$ is a finitely generated $C$-algebra, a contradiction.
\end{proof}

\section{The one-dimensional case}
\label{OneDimensionalCase.sec}

Let $\kk$ be an algebraically closed field. 
The purpose of this section is to classify all extending maximal $\kk$-subalgebras 
of a given one-dimensional affine $\kk$-domain $R$. 
The key ingredient is the following observation. 

\begin{lemma}
	\label{lem:SubringOfOneDimensionalAffineDomain}
	Let $A$ be a $\kk$-subalgebra of the one-dimensional
	affine $\kk$-domain $R$. Then either $A = \kk$ or
	$A$ is a one-dimensional affine $\kk$-domain.
\end{lemma}

\begin{proof}
	We can assume that $A \neq \kk$.
	Then there exists $a \in A \setminus \kk$, which is transcendental over $\kk$.
	By the Krull-Akizuki-Theorem
	applied to $\kk[a]Ê\subseteq A$, it follows that $A$ is
	Noetherian, see for example \cite[Theorem~33.2]{Na1975Local-rings}. 
	By \cite[Corollary~1.2]{OnYo1982On-Noetherian-subr}, 
	we have
	\[
		\dim A = \textrm{tr.deg}_{\kk} A = 1 \, .
	\]
	Let $A'$ be the integral closure of $A$ in its quotient field. By 
	\cite[Theorem~9.3]{Ma1986Commutative-ring-t} it follows that 
	$\dim A' = 1$. In particular, $A'$ is equidimensional. 
	\cite[Theorem~3.2]{OnYo1982On-Noetherian-subr} implies now, 
	that $A$ is an affine $\kk$-domain.
\end{proof}

The next Theorem classifies all extending maximal $\kk$-subalgebras of $R$.

\begin{theorem}
	\label{thm:OneDimensional}
	Take a projective closure $\overline{X}$ of the affine curve $X = \Spec(R)$
	such that $\overline{X}$ is non-singular at every point of
	$\overline{X} \setminus X$ (such an $\overline{X}$ 
	is unique up to isomorphism).
	Let $U \subsetneq \overline{X}$ be a proper open subset that contains $X$
	and the complement $U \setminus X$ is just a single point. Then the image
	of the map on sections $\Gamma(U, \OOO_U) \to R$ is an extending 
	maximal $\kk$-subalgebra of 
	$R$ and every extending maximal $\kk$-subalgebra of $R$ is of this form.
\end{theorem}

\begin{proof}
	First, note that $U \subsetneq \overline{X}$ is an affine curve, see
	\cite[Chp.~IV, Ex.~1.4]{Ha1977Algebraic-geometry}.
	Let $A \subseteq R$ be the image of $\Gamma(U, \OOO_U) \to R$
	and consider an intermediate ring $A \subseteq B \subsetneq R$.
	By Lemma~\ref{lem:SubringOfOneDimensionalAffineDomain}, 
	$B$ is a one-dimensional
	affine $\kk$-domain. Consider the induced maps
	\[
		X \stackrel{f}{\longrightarrow} 
		\Spec(B) \stackrel{}{\longrightarrow} U \, .
	\]
	As this composition is an open immersion, the first map is an open immersion.
	As $f$ is not an isomorphism, the complement $\Spec(B) \setminus f(X)$
	is non-empty. As $U \setminus X$ is a single point, this implies that 
	$\Spec(B) \to U$ is surjective.
	In fact, since $U$ is non-singular in $U \setminus X$, this map is
	an isomorphism and thus we get $A = B$. This proves that $A$ is 
	an extending maximal $\kk$-subalgebra of $R$.
	
	Conversely, let $A \subsetneq R$ be an extending maximal $\kk$-subalgebra.
	By Lemma~\ref{lem:SubringOfOneDimensionalAffineDomain}, 
	$A$ is an affine $\kk$-domain.
	Let $g \colon X \to \Spec(A)$ be the induced map on affine varieties. 
	It is an open immersion and $\Spec(A) \setminus X$ 
	consists only of the crucial maximal ideal $\mm$ of $A$.
	Consider the birational map 
	\begin{equation}
		\label{birmap.eq}
		\Spec(A) \stackrel{g^{-1}}{\dashrightarrow} X \longrightarrow 
		\overline{X} \, ,
	\end{equation}
	which is an open immersion on $g(X)$. We have to show, that this map
	is an open immersion on $\Spec(A)$. By 
	\cite[Proposition~3.3]{FeOl1970Homomorphisms-mini},
	the localization $A_\mm$ is a one-dimensional valuation ring. 
	Since $A_\mm$ is Noetherian, it is a
	discrete valuation ring. Thus $\Spec(A)$ is non-singular at $\mm$
	and therefore the birational map \eqref{birmap.eq} is an injective morphism,
	which is an open immersion locally at $\mm$ (note that $\overline{X}$
	is smooth at every point of $\overline{X} \setminus X$).
	Thus the morphism \eqref{birmap.eq} is an open immersion.
\end{proof}

\begin{example}
	Consider $R = \kk[x, y] / (y-x^3+xy^2)$. The closure of 
	$X = \Spec(R) \subseteq \AA^2$ in 
	$\PP^2$ consists of the three smooth points
	\[
		p_1 = (0: 1: 0) \, , \quad p_2 = (1: 1: 0) \, , \quad p_3 = (-1: 1: 0) \, .
	\]
	The corresponding
	extending maximal $\kk$-subalgebras of $R$ are given by the images 
	on sections of the following maps (all maps are seen as restrictions
	of maps $\AA^2_{x, y} \to \AA^2_{s, t}$)
	\begin{align*}
	p_1 \colon \quad\quad &    X \longrightarrow \{ \, t-s^4+t^2 \, \} \, , \quad 
		&(x, y) &\longmapsto (x, xy) \\
	p_2 \colon \quad\quad &	X \longrightarrow \{ \, s^2 - t + 2t^2-ts^2 \, \}  \, ,
	        \quad &(x, y) &\longmapsto (x-y, (x-y)x) \\ 
	p_3 \colon \quad\quad &	X \longrightarrow \{ \,s^2-t-2t^2+ts^2 \, \} \, , \quad
		&(x, y) &\longmapsto (x+y, (x+y)x) \, .
	\end{align*}
\end{example}

\section{\texorpdfstring{Examples of extending 
			       maximal $\kk$-subalgebras of $\kk[t, y]$ that
			       do not contain a coordinate}
		{Examples of extending maximal k-subalgebras of k[t, y] that
		do not contain a coordinate}}	  
\label{sec.ExamplesContainNotACoordinate}    

It is thus natural to ask,
whether all extending maximal $\kk$-subalgebras of $\kk[t, y]$ contain
a coordinate. In this section we construct plenty of examples, which give
a non-affirmative answer to this question. These examples
indicate that it is difficult to classify all maximal subalgebras
of $\kk[t, y]$. 

For the construction of these examples we use techniques 
of birational geometry of surfaces and the classification
of extending maximal subalgebras of one-dimensional
affine $\kk$-domains. As we fix the algebraically closed field
$\kk$, we write $\PP^n$ for $\PP^n_{\kk}$ and $\AA^n$ for $\AA^n_{\kk}$.

\begin{definition}
	\label{def:TangentOfOrder}
	Let $L \subseteq \PP^2$ be a line, let $p \in L$ be a point and let
	$\Gamma \subseteq \PP^2$ be an irreducible curve with $\Gamma \neq L$
	which passes through $p$. We say that \emph{$\Gamma$ is tangent
	to $L$ at $p$ of order at least $m$}, if there exists a sequence of blow-ups
	\[
		S_m \stackrel{\pi_m}{\longrightarrow} S_{m-1} 
		\stackrel{\pi_{m-1}}{\longrightarrow} \ldots 
		\stackrel{\pi_2}{\longrightarrow}
		S_1 \stackrel{\pi_1}{\longrightarrow} \PP^2
	\]
	such that $\pi_1$ is centered at $p$, 
	$\pi_i$ is centered at a point on the exceptional divisor
	of $\pi_{i-1}$ for $i=2, \ldots, m$ and the strict transforms of $L$ and of 
	$\Gamma$ under $\pi_1 \circ \cdots \circ \pi_m$ 
	have an intersection point on the exceptional divisor
	of $\pi_m$.
\end{definition}

The following Lemma is crucial for our construction.

\begin{lemma}
	\label{lem:Jeremy}
	Let $\Gamma \subseteq \PP^2$ be an irreducible curve
	and let $L \neq \Gamma$ be a line in $\PP^2$.
	Fix some $p \in \Gamma \cap L$. If $\Gamma$ is tangent to $L$ at $p$ 
	of order at least $2$ and if $\Gamma$ is smooth at $p$, then
	there exists no coordinate 
	$f \colon \AA^2 = \PP^2 \setminus L \to \AA^1$ such that
	the rational map $f |_{\Gamma} \colon \Gamma \dashrightarrow \AA^1$
	is defined at $p$.
\end{lemma}

\begin{proof}
	Let $\varphi \colon \PP^2 \dashrightarrow \PP^2$ be a birational map
	that restricts to an automorphism on $\PP^2 \setminus L = \AA^2$.
	Let $\pr \colon \AA^2 \to \AA^1$ be the projection given by $\pr(x, y) = x$.
	We have to prove that the rational map
	$\pr \circ \varphi |_{\Gamma} \colon \Gamma \dashrightarrow \AA^1$
	is not defined at $p$. Let $a \in \PP^2$ be the image of $p$
	under the rational map 
	$\varphi |_{\Gamma} \colon \Gamma \dashrightarrow \PP^2$, 
	which is defined at $p$ since $\Gamma$ is smooth at $p$. 
	We have $a \in L$, since either $\varphi$ contracts the line $L$ to some 
	point on $L$ or $\varphi$ maps $L$ isomorphically onto itself.
	If $a \neq (0:1:0)$, then the map
	$\pr \circ \varphi |_{\Gamma} \colon \Gamma \dashrightarrow \AA^1$ is
	not defined at $p$. Thus we can assume that $a = (0:1:0)$.
	
	Let $\sigma \colon \Bl_a(\PP^2) \to \PP^2$ be the blow-up of 
	$\PP^2$ centered at $a$.
	Then, $\pr \circ \varphi |_{\Gamma} \colon \Gamma \dashrightarrow \AA^1$ 
	is not defined at $p$ if and only if
	\[
		\Gamma \subseteq \PP^2
		\stackrel{\varphi}{\dashrightarrow} \PP^2
		\stackrel{\sigma^{-1}}{\dashrightarrow} \Bl_a(\PP^2)
	\]
	maps $p$ to the intersection point of the exceptional divisor
	of $\sigma$ and the strict transform of $L$ under $\sigma$.
	In other words, we have to prove that $\varphi(\Gamma)$
	is tangent to $L$ at $a$ of order at least $1$.
	
	If $\varphi$ is an automorphism, then
	the result is obvious, so we can assume that there exist 
	base-points of $\varphi$. By \cite[Lemma 2.2]{Bl2009The-correspondence}
	there exist birational morphisms
	$\varepsilon \colon Y \to \PP^2$ and $\eta \colon Y \to \PP^2$
	such that the following is satisfied:
	\begin{itemize}
	\item we have $\eta = \varphi \circ \varepsilon$; 

	\item no curve of self-intersection $-1$ of $Y$ is contracted by both, 
	$\varepsilon$ and $\eta$;	
	
	\item there are decompositions
	\[
		\varepsilon = 
		\varepsilon_1 \circ \cdots \circ \varepsilon_n \colon 
		Y \longrightarrow \PP^2 \quad
		\textrm{and} \quad
		\eta = \eta_1 \circ \cdots \circ \eta_n \colon Y \longrightarrow \PP^2
	\]
	where $\varepsilon_1$ (respectively $\eta_1$) is a blow-up centered at 
	a point on $L$ and  $\varepsilon_i$ (respectively $\eta_i$) is a blow-up
	centered at a point on the exceptional divisor of $\varepsilon_{i-1}$ 
	(respectively $\eta_{i-1}$) for $i > 1$;

	\item the integer $n$ is greater than or equal to $3$;

	\item the strict transform of $L$ 
	under $\varepsilon$ (respectively $\eta$) has self-intersection $-1$.
	\end{itemize}
	
	Let $q_{i-1}$ be the center of $\varepsilon_i$ and let $E_i$ be the exceptional
	divisor of $\varepsilon_i$ for $i = 1, \ldots, n$. Moreover, we denote by 
	$L_i$ the strict transform of $L$ under 
	$\varepsilon_i \circ \cdots \circ \varepsilon_1$ for $i = 1, \ldots, n$.
	Since $(L_n)^2=-1$, we see that $L$ 
	passes through $q_0$, that $L_1$ passes through $q_1$,
	but $L_i$ passes not through $q_i$ for $i > 1$. 

	By assumption, $\Gamma$ is tangent to $L$ at $p$ of order at least
	$2$, so there exists a sequence of blow-ups
	\[
		S_2 \stackrel{\pi_2}{\longrightarrow}
		S_1 \stackrel{\pi_1}{\longrightarrow} \PP^2
	\]	
	such that $\pi_1$ is centered at $p$, $\pi_2$ is centered at
	some point on the exceptional divisor of $\pi_1$ and
	the strict transforms of $L$ and of $\Gamma$ under $\pi_1 \circ \pi_2$
	intersect at one point of the exceptional divisor. 
	Denote this intersection point on $S_2$ by 
	$p_2$. Consider the birational map
	\[
		\psi = 
		\varepsilon_2^{-1} \circ \varepsilon_1^{-1} \circ \pi_1 \circ \pi_2 \, .
	\]
	This map is defined at $p_2$ and we denote by $p_2'$ its image under $\psi$.
	Since $p_2' \in L_2$ and since $q_2 \not\in L_2$ there exists exactly
	one point $r \in L_n$ that is mapped onto $p_2'$ via 
	$\varepsilon_n \circ \cdots \circ \varepsilon_3$ (note that $n \geq 3$).
	Remark that the strict transform of $\Gamma$ under $\varepsilon$
	passes through $r$.
	
	Let $r_{i-1}$ be the center of $\eta_i$ and let $F_i$ be the exceptional
	divisor of $\eta_i$ for $i = 1, \ldots, n$.
	Since $E_n$ and $L_n$ are the only curves of self-intersection $-1$
	lying in $Y \setminus \varepsilon^{-1}(\PP^2 \setminus L)$, 
	it follows that $E_n$
	is the strict transform of $L$ under $\eta$ and that 
	$\eta_n$ contracts $L_n$ i.e. $F_n = L_n$. 
	Hence we have for $i = 2, \ldots, n$
	\[
		\eta_i \circ \cdots \circ \eta_n(r) = r_{i-1} \in F_{i-1} \, .
	\]
	As $r \in L_n$, the curve $L_n$ is contracted by $\eta$ onto $\eta(r)$;
	this point being also the point where $\varphi$ contracts $L$, 
	we get $\eta(r) = a \in L$.
	Since the strict transform of $L$ under $\eta$ has self-intersection $-1$, it 
	follows that $r_1$ is the intersection point of $F_1$ and the strict 
	transform of $L$ under $\eta_1$. As the strict transform of $\Gamma$ 
	under $\varepsilon$ passes through $r$, its image passes through
	all the points $r_i$ and thus also through $r_1$.
	So the curve $\varphi(\Gamma)$ is tangent to $L$ at $a \in \PP^2$ 
	of order at least $1$.
\end{proof}

	With this lemma we can construct plenty of examples
	of extending maximal $\kk$-subalgebras of $\kk[t, y]$ that  do not contain
	a coordinate of $\kk[t, y]$.
	
	Let $X$ be an irreducible curve of $\AA^2$, which is defined by some
	polynomial $f$ in $\kk[t, y]$. 
	Let $\Gamma$ be the closure of $X$ in $\PP^2$. Assume that
	there exists a smooth point $p$ on $\Gamma$ that lies not
	in $X$ and assume that $\Gamma \setminus X$ contains more than one point. 
	Then the ring
	\[
		A = \{ \, h \in \Gamma(X, \OOO_X) \ | \ 
		\textrm{$h$ is defined at $p$} \, \}
	\]
	is an extending maximal $\kk$-subalgebra of $\Gamma(X, \OOO_X)$,
	which is finitely generated over $\kk$, see
	Theorem~\ref{thm:OneDimensional} and 
	Lemma~\ref{lem:SubringOfOneDimensionalAffineDomain}. 
	Let $a_1, \ldots, a_k \in A$ be a set of 
	generators and let $r_1, \ldots, r_k \in \kk[t, y]$ be elements such that
	$r_i |_X = a_i$.
	If $\Gamma$ is tangent to $L = \PP^2 \setminus \AA^2$ at $p$ of order 
	at least $2$, then 
	\[
		\kk[r_1, \ldots, r_k] + f \kk[t, y]
	\]
	is an extending maximal $\kk$-subalgebra 
	of $\kk[t, y]$ that does not contain a coordinate of 
	$\kk[t, y]$, see Lemma~\ref{lem:maxandcond} and Lemma~\ref{lem:Jeremy}.

\section{\texorpdfstring
	     {Classification of maximal subrings of $\kk[t, t^{-1}, y]$
	     that contain $\kk[t, y]$}
	     {Classification of maximal subrings of k[t, t^{-1}, y]
	     that contain k[t, y]} }
\label{ClassContainKty.sec}	
	     
The goal of this section is the classification of all maximal subrings of 
$\kk[t, t^{-1}, y]$ that contain $\kk[t, y]$. Let us start with a simple example.
\begin{example}
	By using Lemma~\ref{lem:maxandcond} one can see that
	\[
		A = \kk[t, y] + (y^2-t) \kk[t, t^{-1}, y]
	\]
	is a maximal subring of $\kk[t, t^{-1}, y]$, which contains $\kk[t, y]$.
	Another description of this ring is the following
	\[
		A = B \cap \kk[t, t^{-1}, y] \, , \
		\textrm{where} \ B = \kk[t^{1/2}, y] + (y-t^{1/2}) \kk[t^{1/2}, t^{-1/2}, y] \, .
	\]
	By using Lemma~\ref{lem:maxandcond}, one can see that $B$
	is a maximal subring of $\kk[t^{1/2}, t^{-1/2}, y]$, which contains
	$\kk[t^{1/2}, y]$. However, the ring $B$ is of a simpler form than $A$
	(we replaced $y^2-t$ by a linear polynomial in $y$).
\end{example}
	     
The general strategy works in a similar way. First we ``enlarge" the coefficients
$\kk[t]$ to some ring $F$ in such a way, that all maximal subrings 
of $F_t[y]$ that contain $F[y]$
have a simple form (in the example, we replaced $\kk[t]$ 
by $F = \kk[t^{1/2}$]). 
Then we prove that the intersection of such a simple maximal subring with
$\kk[t, t^{-1}, y]$ yields a maximal subring of $\kk[t, t^{-1}, y]$ that contains 
$\kk[t, y]$ and that we receive by this intersection-process every maximal subring
that contains $\kk[t, y]$.

For the ``enlargement" of the coefficients we have to introduce some notation and
terminology. 
	     
\subsection{Notation and terminology}

Let $\kk$ be an algebraically closed field (of any characteristic).
We denote by $\kk[[t^{\QQ}]]$ the Hahn field over $\kk$
with exponents in $\QQ$, i.e. the field of all 
formal power series
\[
	\alpha = \sum_{s \in \QQ} a_s t^s	
\]
with coefficients $a_s \in \kk$ and with the property that
the support 
\[
	\supp(\alpha) = \{ \, s \in \QQ \ | \ a_s \neq 0 \, \}
\]
is a well ordered subset of $\QQ$. 
There exists a natural valuation on $\kk[[t^\QQ]]$, namely
\[	
	\nu \colon \kk[[t^{\QQ}]] \longrightarrow \QQ \, , \quad 
	\alpha \longmapsto \min \supp(\alpha) \, .
\]
The valuation ring of $\nu$ we denote by $\kk[[t^{\QQ}]]^+$.
More generally, for any subring $B \subseteq \kk[[t^{\QQ}]]$
we denote by $B^+$ the subring of elements with $\nu$-valuation $\geq 0$,
i.e.
\[
	B^+ = \{ \, b \in B \ | \ \nu(b) \geq 0 \, \} \, .
\]
Finally, for any subring $A \subseteq \kk[[t^{\QQ}]][y]$ 
we denote by $A_1$ the subset of degree one elements, i.e.
\[
	A_1 = \{ \, a \in A \ | \ \deg_y(a) = 1 \, \} \, .
\]

\subsection{Organisation of the section}
In Subsection~\ref{subsec:biggclass}, 
we classify all maximal subrings of $K[y]$ that contain
$K^+[y]$ for any algebraically closed field $K \subseteq \kk[[t^{\QQ}]]$
that contains the field of rational functions $\kk(t)$ and satisfies the so called
\emph{cutoff-property} (see Definition~\ref{cutoff.def}). For example, the Hahn
field $\kk[[t^{\QQ}]]$, the Puiseux field $\bigcup_n \kk((t^{1/n}))$
or the algebraic closure of $\kk(t)$ enjoy the cutoff-property 
(see Example~\ref{exp:cutoff}).

In Subsection~\ref{subsec:descriptionByIntersection}, 
we prove that for any maximal subring $A \subsetneq K[y]$
containing $K^+[y]$, the intersection 
$A \cap \kk[t, t^{-1}, y]$ is again a maximal subring of $\kk[t, t^{-1}, y]$.
Moreover, we prove that any maximal subring of $\kk[t, t^{-1}, y]$ that contains
$\kk[t, y]$ can be constructed as an intersection like above.

Thus we are left with the question, which of the maximal subrings
of $K[y]$ that contain $K^+[y]$ give the same ring, after intersection with 
$\kk[t, t^{-1}, y]$. We give an answer to this question in 
Subsection~\ref{subsec:Class}.

\subsection{Classification of maximal 
subrings of $K[y]$ that contain $K^+[y]$}
\label{subsec:biggclass}
Throughout this subsection, we fix an algebraically closed subfield 
$K \subseteq \kk[[t^{\QQ}]]$ that contains the field of rational functions $\kk(t)$.


\begin{proposition}
	\label{prop:rep}
	Let $K^+[y] \subseteq A \subsetneq K[y]$ be an intermediate ring
	and assume that $A \subsetneq K[y]$ satisfies the property $P_2$ in $K[y]$.
	Then $A = K^+[A_1]$.
\end{proposition}

\begin{proof}
	Let $a \in A$.
	After multiplying $a$ with a unit of $K^+$, we can assume that
	\[
		a = \frac{y^n}{t^s} + \textrm{lower degree terms in $y$} \, ,
	\]
	where $s \in \QQ$ and $n \geq 0$ is an integer. We have to show that
	$a \in K^+[A_1]$.
	We proceed by induction on $n$. If $n = 0$, then $a \in A \cap K = K^+$.
	So let us assume $n > 0$.
	As $K$ is algebraically closed
	and contains $t$, there exist 
	$\alpha_1, \dots, \alpha_n \in K$ with
	\[
		a = \left(\frac{y-\alpha_1}{t^{s/n}}\right) 
		      \left(\frac{y-\alpha_2}{t^{s/n}}\right) \cdots
		      \left(\frac{y-\alpha_n}{t^{s/n}}\right) \, .
	\]
	Since $A \subseteq K[y]$ satisfies the property $P_2$, we have
	$(y - \alpha_i) / (t^{s/n}) \in A$ for some $i$. This implies that 
	\[
		\frac{(y - \alpha_i)^n}{t^s} \in 
		K^+[A_1] \, .
	\]
	Thus $q = a - (y - \alpha_i)^n / t^s \in A$. By induction hypothesis we have
	$q \in K^+[A_1]$ and thus $a \in K^+[A_1]$. Hence $A \subseteq K^+[A_1]$,
	which implies the result.
\end{proof}

\begin{lemma}
	\label{lem:P2}
	Let $K^+[y] \subseteq E \subsetneq K[y]$ be a proper subring. 
	Then there exists a
	proper subring $E' \subsetneq K[y]$ that satisfies the property $P_2$ 
	and contains $E$.
\end{lemma}

\begin{proof}
	Denote by $\tilde{E} \subseteq K[y]$ the integral closure
	of $E$ in $K[y]$. As $E \neq K[y]$, it follows that 
	$tE$ is a proper ideal of $E$. In particular,
	$\varphi \colon \Spec(K[y]) \to \Spec(E)$ is 
	nonsurjective. Since $\Spec(\tilde{E}) \to \Spec(E)$ is surjective
	(see \cite[Theorem~9.3]{Ma1986Commutative-ring-t}), it follows that
	$\tilde{E} \neq K[y]$. Hence there exists an intermediate ring
	$\tilde{E} \subseteq E' \subsetneq K[y]$ that satisfies the property $P_2$
	in $K[y]$, by \cite[Th\'eor\`eme~8]{Sa1957La-notion-de-place}.
\end{proof}

Now, we give an application of these two results to maximal subrings.
Roughly speaking, the proposition says, that for rings which are generated
by degree one elements, one can see the maximality already 
on the level of degree one elements.

\begin{proposition}
	\label{prop:red}
	Let $K^+[y] \subseteq A \subsetneq K[y]$ be a proper subring 
	that satisfies $A = K^+[A_1]$. Then $A$ is maximal in $K[y]$ if and only if
	\begin{equation}
		\label{eq:DegOne}
		\textrm{for all $f \in K[y] \setminus A$ of degree $1$ we have 
		$A[f] = K[y]$} \, .
	\end{equation}
\end{proposition}

\begin{proof}	
	Assume that $A$ satisfies \eqref{eq:DegOne}.
	Let $A \subseteq E \subsetneq K[y]$ be an intermediate ring.
	We want to prove $A = E$. By Lemma~\ref{lem:P2}, there
	exists a proper subring 
	$E' \subsetneq K[y]$ that satisfies the property $P_2$ and contains $E$.
	Now, if there would exist $f \in (E')_1 \setminus A_1$, then
	we would have by \eqref{eq:DegOne}
	\[
		K[y] = A[f] \subseteq E' \subseteq K[y] \, .
	\]
	This would imply that $E' = K[y]$, a contradiction. Thus we have
	$A_1 = (E')_1$. According to Proposition~\ref{prop:rep} we have 
	$A = E'$ and therefore $A = E$.
	
	The other implication is clear.
\end{proof}

\begin{definition}
	Let $S = \{ \, s_1 < s_2 < \ldots \, \}$ be a strictly monotone sequence in
	$\QQ_{\geq 0}$ and let 
	$\Lambda = \{ \, \alpha_1, \alpha_2, \ldots \, \}$ be a sequence in 
	$K$ such that $\supp(\alpha_i) \subseteq [0, s_i)$ and
	$\supp(\alpha_{i+1}- \alpha_{i}) \subseteq [s_i, s_{i+1})$ for all $i > 0$.
	We call then $(S, \Lambda)$ an \emph{admissible pair of $K$}.
	If $\alpha$ is an element of $K$ such that 
	$\supp(\alpha - \alpha_i) \subseteq [s_i, \infty)$ for all $i > 0$, then we
	call $\alpha$ a \emph{limit} of the admissible pair $(S, \Lambda)$.
\end{definition}


\begin{lemma}
	\label{lem:Limit}
	Let $(S, \Lambda)$ be an admissible pair of $\kk[[t^\QQ]]$. 
	Then there exists a limit
	in $\kk[[t^\QQ]]$.
	Moreover, if $\lim s_i = \infty$, then $\alpha$ is unique.
\end{lemma}

\begin{proof}
	Let $\alpha_i = \sum a_{is} t^s$.
	Now, we define
	$\alpha = \sum a_s t^s$, where
	$a_s = a_{is}$ for some $i$ with $s_i > s$.
	One can easily check, that $a_s$ is well defined. Moreover,
	\[
		\supp(\alpha) = \bigcup_{i = 1}^\infty \supp(\alpha_i) 
		\quad \textrm{and} \quad
		\supp(\alpha) \cap [0, s_i)  = \supp(\alpha_i) \ 
		\textrm{for $i  = 1, 2, \ldots$}
	\]
	and thus $\supp(\alpha)$ is well ordered. It follows that 
	$\alpha \in \kk[[t^\QQ]]^+$
	and that $\supp(\alpha - \alpha_i) \subseteq [s_i, \infty)$ for all $i > 0$.
	The uniqueness statement is clear.
\end{proof}

\begin{definition}
	\label{cutoff.def}
	The subfield $K�\subseteq \kk[[t^\QQ]]$ satisfies the \emph{cutoff property},
	if for all $\alpha = \sum_s a_s t^s \in K$ and for all 
	$u \in \QQ$ we have $\sum_{s \ge u} a_s t^s \in K$.
\end{definition}

\begin{example}
	\label{exp:cutoff}
	An important example of an algebraically closed field 
	inside $\kk[[t^\QQ]]$ that contains $\kk(t)$ and satisfies the cutoff
	property is the Puiseux field
	\[
		\bigcup_{n = 1}^\infty \kk((t^{1/n})) \, .
	\]
	Clearly, the Hahn field $\kk[[t^\QQ]]$ itself is an example.
	Another example is the algebraic closure of $\kk(t)$ (inside the Puiseux field). 
	This follows from the fact that $\sum_{i = -m}^m a_i t^{i/n}$ is algebraic
	over $\kk(t)$ where $a_i \in \kk$ and $n, m \in \NN$ (it is the sum of algebraic
	elements).
\end{example}

In the next proposition we classify all $P_2$-subrings of $K[y]$ under the additional
assumption, that $K$ satisfies the cutoff property.

\begin{proposition}
	\label{prop:P_2}
	Assume that $K$ satisfies the cutoff property. Let
	$K^+[y] \subseteq E \subsetneq K[y]$. Then
	$E$ satisfies the property $P_2$ in $K[y]$ if and only if
	\[
		E = K^+\left[ \left\{ \,  \frac{y-\alpha_i}{t^{s_i}} \ \big| \
		 i = 1, 2, \ldots \, \right\} \right]
	\]
	for an admissible pair $(S, \Lambda)$ of $K$ or $E = K^+[e]$ for some 
	element $e \in K[y]$ of degree $1$.
\end{proposition}

\begin{proof}
	Assume that $K^+[y] \subseteq E \subsetneq K[y]$ is a subring that satisfies
	the property $P_2$ in $K[y]$. We consider the following subset of $E_1$:
	\[
		N = 
		\left \{ \, \frac{y-\alpha}{t^s} \in E \ | \ 
			     \textrm{$s \in \QQ_{>0}$, $\alpha \in K$ 
			     and $\supp(\alpha) \subseteq [0, s)$} \, \right \} \, .
	\]
	Using the fact, that $E \cap K = K^+$ and that $K$ satisfies the cutoff property,
	one can see that $N$ has the following two properties:
	\begin{enumerate}[i)]
		\item If $(y-\alpha) / t^s, (y-\alpha') / t^{s'} \in N$
		and $s \le s'$, then $\supp(\alpha' - \alpha) \subseteq [s, s')$.
		\item If $(y-\alpha)/ t^s \in E$, $s \in \QQ_{>0}$ 
		and $\alpha \in K$, then $\alpha \in K^+$ and there exists $n \in N$,
		such that $((y-\alpha)/t^s) - n \in K^+$.
	\end{enumerate}
	Property ii) of $N$ implies
	\begin{equation}
		\label{eq:equality}
		K^+[N, y] = K^+[E_1] \, .
	\end{equation}
	
	Let $U \subseteq \QQ_{>0}$ be the set of all $s \in \QQ_{> 0}$ such that
	there exists $\alpha \in K^+$ with $(y-\alpha)/t^s \in N$.
	Property i) of $N$ implies that 
	for every $s \in U$ there exists a unique
	$\alpha_s \in K^+$ such that $(y-\alpha_s)/t^s \in N$.
	Now, we make the following distinction.
	\begin{enumerate}[\qquad \qquad, \quad]
		\item[$\sup(U) \in U$:] Let $u = \sup(U)$. It follows from property i)
			of $N$, that $(y-\alpha_s)/t^s \in K^+[(y-\alpha_u)/t^u]$ for all
			$s \in U$. This implies $K^+[N, y] \subseteq K^+[(y-\alpha_u)/t^u]$.
			Clearly, we have $K^+[(y-\alpha_u)/ t^u] \subseteq K^+[N, y]$.
			With \eqref{eq:equality} and Proposition~\ref{prop:rep}, we
			get the equality $E = K^+[(y-\alpha_u)/t^u]$.
		\item[$\sup(U) \notin U$:] 
		 	Let $S = \{ \, s_1 < s_2 < \ldots \, \}$ 
			be a sequence in $U$ such that $\lim s_i = \sup(U)$.
			If we set $\alpha_i = \alpha_{s_i}$ and 
			$\Lambda = \{ \, \alpha_1, \alpha_2, \ldots \, \}$, 
			then $(S, \Lambda)$
			is an admissible pair. Let $s \in U$. As $\sup(U) \notin U$,
			there exists $i$ with $s_i > s$. With property i) of $N$,
			we get now $(y-\alpha_s) / t^s \in K^+[(y-\alpha_i)/t^{s_i}]$.
			Thus $K^+[N, y]$ is generated over $K^+$ by $(y-\alpha_i)/t^{s_i}$,
			$i = 1, 2, \ldots$ .
			By \eqref{eq:equality} and Proposition~\ref{prop:rep} we get
			$K^+[N, y] = K^+[E_1] = E$. 
	\end{enumerate}
	Thus $E$ has the claimed form.
	
	\smallskip
	
	Now, we prove that $K^+[e]$ satisfies the property $P_2$ in $K[y]$, 
	provided that $e \in K[y]$ has degree
	$1$. By applying a $K$-algebra automorphism of $K[y]$, we can assume
	that $e = y$. Consider the following extension of the valuation $\nu |_K$
	on $K$ to $K[y]$
	\[
		\mu \colon K[y] \longrightarrow \QQ \, , 
		\quad f_0 + \ldots + f_n y^n \longmapsto 
		\min \{ \nu(f_0), \ldots, \nu(f_n)�\} \, ,
	\]
	which extends (uniquely) to $K(y)$. Then $K^+[y]$ is exactly
	the set of elements in $K[y]$ with $\mu$-valuation $\geq 0$. 
	From this it follows readily that $K^+[y]$ satisfies the property $P_2$ in $K[y]$.
	
	Now, let $(S, \Lambda)$ be an admissible pair of $K$. Then
	\[
		K^+\left[\left\{ \, \frac{y-\alpha_i}{t^{s_i}} \ \big| 
		\ i = 1, 2, \ldots \, \right\}\right]
	\]
	satisfies the property $P_2$ in $K[y]$ as it is the union of the increasing
	$P_2$-subrings
	\[
		K^+\left[\frac{y-\alpha_1}{t^{s_1}}\right]
		\subseteq K^+\left[\frac{y-\alpha_2}{t^{s_2}}\right] \subseteq \cdots
		\, .	
	\]
\end{proof}

With this classification result at hand, we can now achieve a 
classification of all maximal subrings of $K[y]$ that contain $K^+[y]$. 

\begin{proposition}
	\label{prop:class}
	Assume that $K$ satisfies the cutoff property.
	Let $(S, \Lambda)$ be an admissible pair of $K$.
	Assume that either $\lim s_i = \infty$ or $(S, \Lambda)$ has no limit in $K$.
	Then
	\begin{equation}
		\label{eq:MaxGen}
		K^+\left[ \left\{ \,  \frac{y-\alpha_i}{t^{s_i}} \ \big| \
		 i = 1, 2, \ldots \, \right\} \right]
	\end{equation}
	is a maximal subring of $K[y]$ that contains $K^+[y]$. On the other hand,
	every maximal subring of $K[y]$ that contains $K^+[y]$ is of this form.
\end{proposition}

\begin{proof}
	Let $B \subseteq K[y]$ be the ring of \eqref{eq:MaxGen}. 
	We claim that $B \neq K[y]$. Otherwise, there exists $i$ such that
	$1/t \in K^+[ (y-\alpha_i) / t^{s_i}]$, as
	$(S, \Lambda)$ is an admissible pair. 
	This would imply $1/t \in K^+$, a contradiction.
	
	Note, that we have $B = K^+[B_1]$.
	Thus, according to Proposition~\ref{prop:red} it is enough to show,
	that $B[f] = K[y]$ for all $f \in K[y] \setminus B$ of degree $1$. 
	Up to multiplying $f$
	with a unit of $K^+$, we can assume that $f = (y - \alpha) / t^s$
	for some $\alpha \in K$ and $s \in \QQ$.
	First, assume that $s < \lim s_i$. Hence there exists $i$ with $s < s_i$.
	Thus we have
	\[
		\frac{\alpha_i -\alpha}{t^s} = 
		\frac{y -\alpha}{t^s} - \frac{y-\alpha_i}{t^s} \in K[y] \setminus B \, .
	\]
	So this last element lies in $K \setminus K^+$. Hence we have 
	$B[f] = K[y]$. Now, assume $s \ge \lim s_i$ 
	(and thus $\lim s_i$ is finite). As $(S, \Lambda)$ has no limit in $K$, 
	there exists $i$ such that $\supp(\alpha-\alpha_i)$ is not contained in
	$[s_i, \infty)$. Thus,
	\[
		 \frac{y -\alpha}{t^{s_i}} - \frac{y-\alpha_i}{t^{s_i}}
		 = \frac{\alpha_i -\alpha}{t^{s_i}} \in K \setminus K^+ \, ,
	\]
	and hence we get $B[f] = K[y]$ again.
	
	\smallskip
	
	Now, let $A \subsetneq K[y]$ be a maximal subring that contains $K^+[y]$,
	which must be an extending maximal subring.
	By Lemma~\ref{lem:PropertyP_2ForExtMax},
	$A \subseteq K[y]$ satisfies the property $P_2$.  
	Since $K^+[e] \subseteq K[y]$ is not a maximal subring 
	for all $e \in K[y]$ of degree $1$,
	it follows from Proposition~\ref{prop:P_2} that there exists an admissible pair
	$(S, \Lambda)$, such that 
	\[
		A = K^+\left[ \left\{ \,  \frac{y-\alpha_i}{t^{s_i}} \ \big| \
		 		i = 1, 2, \ldots \, \right\} \right] \, .
	\]
	It remains to prove that $\lim s_i = \infty$ or $(S, \Lambda)$ 
	has no limit in $K$.
	Assume towards a contradiction that $s = \lim s_i < \infty$ and
	$\alpha \in K$ is a limit of $(S, \Lambda)$. Then, it follows that
	$A \subseteq K^+[(y-\alpha)/t^s]$. As $K^+[(y-\alpha)/t^s]$
	is certainly not a maximal subring of $K[y]$, we get
	a contradiction. This finishes the proof.
\end{proof}

\begin{remark}
	\label{rem:class}
	Let $A$ be the maximal subring $\eqref{eq:MaxGen}$ in 
	the Proposition~\ref{prop:class}. We describe the crucial maximal ideal of $A$. 
	Let $\nn \subseteq K^+$ be the unique maximal ideal. In fact,
	$\nn = \sum_{q \in \QQ_{> 0}} t^q K^+$.
	For $i \in \NN$, let $\alpha'_i = \alpha_i + t^{s_i} a_{i+1}$ where
	$a_{i+1} \in \kk$ denotes the coefficient of $t^{s_i}$ in $\alpha_{i+1}$.
	Thus $A$ is generated over $K^+$ by the elements $(y-\alpha'_i)/ t^{s_i}$.
	We have the following inclusion of ideals in $A$
	\[
		\nn + \sum_{i = 0}^\infty \frac{y-\alpha'_i}{t^{s_i}} A
		\subseteq \sum_{q \in \QQ_{>0}} t^q A \, .
	\]
	As every element of $A$ is an element of $\kk \cdot 1$ 
	modulo the left hand ideal
	and the right hand ideal is proper in $A$, these ideals are the same. 
	It follows, that this ideal is maximal, has residue field $\kk$
	and it is the crucial maximal ideal.
\end{remark}


\begin{proposition}
	\label{prop:conductor}
	For $\alpha \in K^+$, the ring $K^+ + (y-\alpha) K[y]$ is a maximal subring
	of $K[y]$ that contains $K^+[y]$, 
	with non-zero conductor ideal $(y-\alpha) K[y]$. 
	Moreover, all maximal subrings 
	$K^+[y] \subseteq A \subsetneq K[y]$ with non-zero conductor are of this form. 
\end{proposition}


\begin{proof}
	The first statement follows from Lemma~\ref{lem:maxandcond}. 
	For the second statement, let 
	$K^+[y] \subseteq A \subsetneq K[y]$ be a maximal subring and assume
	there exists $0 \neq f \in A$ such that $f K[y] \subseteq A$.
	We can assume that $f$ is monic in $y$.
	Let $f = f_1 \cdots f_k$ be the decomposition of $f$ into monic linear factors
	inside $K[y]$. As $A \subseteq K[y]$ satisfies the property
	$P_2$, for all $n \in \NN$ there exists $i = i(n)$ such that $f_i / (t^{n/k}) \in A$.
	This implies that there exists $i$ such that $f_i / t^n \in A$ for all 
	$n \in \NN$. Let $f_i = y - \alpha_i$. Hence, 
	$K^+[y] + (y-\alpha_i) K[y] \subseteq A$. Since $A \subsetneq K[y]$ is a proper
	subring, we get $\alpha_i \in K^+$ and thus $A = K^+ + (y-\alpha_i) K[y]$.
\end{proof}

\begin{remark}
	\label{rem:valuation}
	Assume that $K$ satisfies the cutoff property.
	Let $A \subseteq K[y]$ be a maximal subring that contains $K^+[y]$,
	$I \subseteq K[y]$ the conductor ideal of $A$ in $K[y]$ 
	(which could be zero), and let 
	$\mm \subseteq A$ be the crucial maximal ideal.
	By~\cite[Proposition 3.3]{FeOl1970Homomorphisms-mini}, the localization
	$(A / I)_{\mm}$ is a one-dimensional valuation ring. Let $(S, \Lambda)$ be
	an admissible pair in $K$ such that $A$ is generated over $K^+$ by 
	$(y-\alpha_i)/ t^{s_i}$ for $i = 1, 2, \ldots \,$, see Proposition~\ref{prop:class}. 
	Let $\alpha \in \kk[[t^{\QQ}]]^+$
	be a limit of $(S, \Lambda)$, which is not unique (however, it exists
	by Lemma~\ref{lem:Limit}). Using Proposition~\ref{prop:class} and 
	Proposition~\ref{prop:conductor} one can check that the $K$-homomorphism
	\[
		K[y] / I \longrightarrow \kk[[t^\QQ]] \, , \quad f \longmapsto f(\alpha)
	\]
	is injective. Hence, 
	\[
		\omega \colon Q(K[y] / I) \longrightarrow \QQ \, , \quad
		f \longmapsto \nu(f(\alpha))
	\]
	is a valuation on the quotient field of $K[y] / I$. With the aid of 
	Remark~\ref{rem:class} one can see, that the valuation on $(A/I)_{\mm}$
	is given by $\omega$. In particular we have for $f \in K[y]$
	\[
		f \in A \quad \Longleftrightarrow \quad \omega(\bar{f}) \geq 0 \, ,
	\]
	where $\bar{f}$ denotes the residue class modulo $I$. Moreover,
	we get for the crucial maximal ideal
	\[
		f \in \mm \quad \Longleftrightarrow \quad \omega(\bar{f}) > 0 \, .
	\]
	This characterization of $A$ and $\mm$ will be very important for us.
\end{remark}

As a consequence of Proposition~\ref{prop:class} and Lemma~\ref{lem:Limit}
we can now classify all the maximal subrings of $\kk[[t^\QQ]][y]$ which contain
$\kk[[t^\QQ]]^+[y]$.

\begin{corollary}
	\label{cor:HahnClass}
	If $K = \kk[[t^\QQ]]$, then for all $\alpha \in K^+$ the ring
	\[
		K^+\left[ \left\{ \,  \frac{y-\alpha}{t^s} \ \big| \
		s = 1, 2, \ldots \, \right\} \right] \,
	\]
	is maximal in $K[y]$ and contains $K^+[y]$. On the other hand, every
	maximal subring of $K[y]$ that contains $K^+[y]$ is of this form.
\end{corollary}

\begin{remark}
	\label{rem:HahnClass}
	The maximal subring of $K[y]$ in 
	Corollary~\ref{cor:HahnClass} is 
	the ring $K^+ + (y-\alpha)K[y]$. Its crucial maximal ideal is
	$\nn + (y-\alpha)K[y]$, where $\nn \subseteq K^+$ denotes the unique
	maximal ideal.
\end{remark}

With Proposition~\ref{prop:class} and Proposition~\ref{prop:conductor} 
at hand, we can now give another description
of the maximal subrings of $K[y]$ that contain $K^+[y]$ in the case where 
$K$ is the algebraic closure of $\kk(t)$. We just want to stress the 
following definition in advance.

\begin{definition}
	\label{def:StrictlyIncreasingSequence}
	A subset $S$ of $\QQ$ is called a \emph{strictly increasing sequence}
	if there exists an isomorphism of the natural numbers to $S$
	that preserves the given orders.
\end{definition}

\begin{proposition}
	\label{prop:AlgClorClass}
	Let $K$ be the algebraic closure of $\kk(t)$ inside $\kk[[t^\QQ]]$ and let 
	$\mathscr{S}$ be the set of $\alpha \in \kk[[t^\QQ]]^+$
	such that $\supp(\alpha)$ is contained in a strictly increasing sequence.
	Then we have bijections	
	\begin{align*}
		\Xi_1 \colon K^+ & \ \longrightarrow \
		\left\{
		    \begin{array}{c} 
			\textrm{maximal subrings of $K[y]$ with} \\ 
			\textrm{non-zero conductor that contain $K^+[y]$}		    
		   \end{array}
		\right\} \\
		\Xi_2 \colon \mathscr{S} \setminus K^+ & \ \longrightarrow \
		\left\{
		    \begin{array}{c} 
			\textrm{maximal subrings of $K[y]$ with} \\ 
			\textrm{zero conductor that contain $K^+[y]$}
		    \end{array}
		\right\}
	\end{align*}
	given by 
	\[
		\Xi_1(\alpha) = K^+ + (y-\alpha)K[y]
		\quad \textrm{and} \quad 
		\Xi_2(\beta) = K^+\left[ \left\{ \,  \frac{y-\beta_i}{t^{s_i}} \ \big| \
		 		i \in \NN \, \right\} \right]
	\]
	where $\{ \, s_1 < s_2 <  \ldots \, \} = \supp(\beta)$ and
	$\beta_i$ is the sum of the first $i-1$ non-zero terms of $\beta$.
\end{proposition}

\begin{proof}
	Proposition~\ref{prop:conductor} implies that $\Xi_1$ is bijective.
	
	Let $\beta \in \mathscr{S} \setminus K^+$ 
	and let $S = \{ \, s_1 < s_2 < \ldots \, \}$,
	$\Lambda = \{ \, \beta_1, \beta_2, \ldots \, \}$. Then $(S, \Lambda)$
	is an admissible pair and $\beta$ is a limit of it. Since 
	$\beta \not\in K^+$ and since $K$ satisfies the cutoff property, 
	there exists no limit of $(S, \Lambda)$ in $K^+$. Hence, by 
	Proposition~\ref{prop:class} the subring $\Xi_2(\beta)$ is maximal in $K[y]$.
	Thus $\Xi_2$ is well-defined. 
	
	Let $A \subseteq K[y]$ be a maximal subring with zero conductor that
	contains $K^+[y]$. By Proposition~\ref{prop:class} there 
	exists an admissible pair $(S', \Lambda')$ in $K$ such that
	\[
		A = K^+\left[ \left\{ \,  \frac{y-\beta'_i}{t^{s'_i}} \ \big| \
		 		i \in \NN \, \right\} \right]
	\]
	where $S' = \{ \, s'_1 < s'_2 <  \ldots \, \}$ and $\Lambda' = 
	\{ \, \beta'_1, \beta'_2, \ldots \, \}$, and either $\lim s'_i = \infty$
	or $(S', \Lambda')$ has no limit in $K$. If $(S', \Lambda')$ has a limit 
	in $K$, then the conductor of $A \subseteq K[y]$ is non-zero.
	Thus $(S', \Lambda')$ has no limit in $K$. Since $K$ is the algebraic closure
	of $\kk(t)$, the support $\supp(\beta'_i)$ is finite
	for all $i$. Hence the pair $(S', \Lambda')$ has a limit $\beta'$ inside 
	$\mathscr{S} \setminus K^+$. Moreover, this limit satisfies 
	$\Xi_2(\beta') = A$, which proves the surjectivity of $\Xi_2$.
	
	Let $\gamma_1$, $\gamma_2 \in \mathscr{S} \setminus K^+$ such that
	$\Xi_2(\gamma_1) = \Xi_2(\gamma_2)$ and denote this ring by $D$. 
	For $k=1,2$, let 
	$\{ \, s_{k1} < s_{k2} <  \ldots \, \} = \supp(\gamma_k)$
	and let $\gamma_{ki} \in K$ be the 
	sum of the first $i-1$ non-zero terms of $\gamma_k$. Let 
	$i > 0$ be an integer. Without loss of generality we can assume that 
	$s_{1i} \leq s_{2i}$. Since
	$(y-\gamma_{ki})/t^{s_{ki}} \in D$ for $k = 1, 2$, it follows that
	\[
		\frac{\gamma_{2i}- \gamma_{1i}}{t^{s_{1i}}} = 
		\frac{y-\gamma_{1i}}{t^{s_{1i}}} - \frac{y-\gamma_{2i}}{t^{s_{1i}}}
		\in D \cap K = K^+ \, .
	\]
	Hence $\gamma_{2i} = \gamma_{1i} + t^{s_{1i}} \eta$ where $\eta \in K^+$.
	However, since $\supp(\gamma_{1i})$ and $\supp(\gamma_{2i})$ have the
	same number of elements, it follows that $\eta = 0$. Thus
	$\gamma_{1i} = \gamma_{2i}$ for all $i$. This implies that 
	$\gamma_1 = \gamma_2$ and hence $\Xi_2$ is injective.
\end{proof}

\subsection{\texorpdfstring{Description of all maximal subrings of $\kk[t, t^{-1}, y]$ 
	that contain $\kk[t, y]$ by ``intersection"}
	{Description of all maximal subrings of k[t, t^{-1}, y] 
	that contain k[t, y] by "intersection"} }
\label{subsec:descriptionByIntersection}

In this subsection we still fix an algebraically closed subfield 
$K \subseteq \kk[[t^{\QQ}]]$ that contains the field of rational functions $\kk(t)$.
Moreover, we fix a subring $L \subseteq K$ that contains $\kk[t, t^{-1}]$. 
%
Recall that $L^+$ (respectively $K^+$) denotes the elements in 
$L$ (respectively $K$) of $\nu$-valuation $\geq 0$.

\begin{lemma}
	\label{lem:flat}
	The ring extension $L^+ \subseteq K^+$ is flat.
\end{lemma}

\begin{proof}
	Let $\nn$ be the unique maximal ideal of the 
	valuation ring $K^+$. This ideal consists of all elements in $K$
	with $\nu$-valuation $> 0$. Denote by $L'$ the localization
	$(L^+)_{\nn \cap L^+}$. We show that $K^+$ is a flat
	$L'$-module, which implies then the result. 
	Clearly, $K^+$ is a torsion-free $L'$-module. By 
	\cite[Chp. I, \S2, no. 4, Proposition~3]{Bo1972Elements-of-mathem},
	it is thus enough to prove that $L'$ is a valuation ring. 
	
	Let $g, h \in L^+$ and assume that $g \neq 0$,  $\nu(h/g) \geq 0$. 
	As the value group of $\nu$ is $\QQ$,
	there exist integers $a \geq 0$, $b > 0$ such that $\nu(g) = a/b$. 
	Thus we get 
	\[
		\nu\left(\frac{h g^{b-1}}{t^a}\right) \geq 0
		\quad \textrm{and} \quad \nu\left(\frac{g^b}{t^a}\right) = 0
	\]
	and therefore $hg^{b-1}/t^a \in L^+$, $g^b/t^a \in L^+ \setminus \nn$.
	This implies $h/g \in L'$. Hence, $L'$ is a valuation ring 
	(with valuation $\nu |_{Q(L^+)}$).
\end{proof}

%

Our first result says that one can construct every maximal subring
of $L[y]$ that contains $L^+[y]$ by intersecting $L[y]$ with 
some maximal subring of $K[y]$ that contains $K^+[y]$ under a certain assumption.

\begin{proposition}
	\label{prop:existenceAbove}
	Assume that $L^+$ is a maximal subring of $L$.
	If $L^+[y] \subseteq B \subsetneq L[y]$ is a maximal subring,
	then there exists a maximal subring $K^+[y] \subseteq A \subsetneq K[y]$
	such that $B = A \cap L[y]$.
\end{proposition}

\begin{remark}
	The assumption, that $L^+$ is a maximal subring of $L$
	is satisfied for example if $L = \kk[t, t^{-1}]$ or $L = \kk(t)$.
\end{remark}

\begin{proof}

	Let $M$ be the $L^+$-module $L[y] / B$.
	By assumption, $M$ is non-zero. In fact, since $L^+$ is a maximal subring
	of $L$, we have an injection
	\[
		L^+ / tL^+ 
		\longrightarrow M \, , \quad \lambda \longmapsto \lambda t^{-1} \, .
	\]
	Since $K^+$ is a flat $L^+$-module (see Lemma~\ref{lem:flat}),
	we get an injection
	\[
		K^+ / t K^+ \simeq K^+ \otimes_{L^+} (L^+ / tL^+) 
		\longrightarrow K^+ \otimes_{L^+} M \, .
	\]
	Thus $K^+ \otimes_{L^+} M$ is non-zero. Again, since 
	$K^+$ is a flat $L^+$-module, this implies that
	\[
		K^+ \otimes_{L^+} B \subsetneq K^+ \otimes_{L^+} L[y] \, .
	\]
	Therefore, $K^+[B]$ is a proper subring of $K[y]$, which contains
	$K^+[y]$. Applying Zorn's Lemma to
	\[
		\{ \, A \subseteq K[y] \ | \
		A \supseteq K^+[B] \ \textrm{and $t^{-1} \not\in A$} \,Ê\}
	\]
	yields a maximal subring $A$ in $K[y]$ that lies over
	$K^+[B]$. Thus $B = A \cap L[y]$.	
\end{proof}

%

%

In the next proposition we prove that any maximal subring of $K[y]$ that lies
over $K^+[y]$ gives a maximal subring of $L[y]$ after intersection with $L[y]$.

\begin{proposition}
	\label{prop:interscting}
	Assume that $K$ satisfies the cutoff property.
	Let $K^+[y] \subseteq A \subsetneq K[y]$ be a maximal subring
	and let $B = A \cap L[y]$.
	Then
	\begin{enumerate}[i)]
		\item If $I$ denotes the conductor ideal of 
			$A$ in $K[y]$, then $I \cap B$ is the conductor ideal of
			$B$ in  $L[y]$.
		\item The subring $B \subsetneq L[y]$ is maximal. Moreover,
			if $\mm$ denotes the crucial maximal ideal of $A$,
			then $\mm \cap L[y]$ is the crucial maximal ideal of $B$.
	\end{enumerate}
\end{proposition}

\begin{proof} $ $
	\begin{enumerate}[i)]
		\item Let $b \in I \cap B$. Then $b L[y] \subseteq A \cap L[y] = B$.
			Thus $b$ lies in the conductor of $B$ in $L[y]$.
			Now, let $f \in B$ be an element of the conductor
			of $B$ in $L[y]$. Then we have $f L[y] \subseteq B$ 
			and in particular, $f/ t^n \in B \subseteq A$ for all $n \in \NN$.
			As $K[y] = A_t$, this implies that $f K[y] \subseteq A$.
			Thus $f \in I \cap B$.
		\item	Let $I \subseteq K[y]$ be the conductor ideal of $A$ in $K[y]$.
			By i) the intersection $J = I \cap B$ is the conductor ideal
			of $B$ in $K[y]$. 
			Let $\mm \subseteq A$ be the crucial
			maximal ideal and let $\nn = \mm \cap B$. 
			We divide the proof in several steps
			\begin{enumerate}[a)]
			\item 
			We claim that $(B/J)_{\nn}$ is a one-dimensional
			valuation ring. Since $(A/I)_\mm$ is a one-dimensional
			valuation ring 
			(see \cite[Proposition 3.3]{FeOl1970Homomorphisms-mini}), 
			it is enough to prove that
			\[
				(B/J)_{\nn} = 
				(A/I)_{\mm} \cap Q(L[y]/J)
			\]
			inside $Q(K[y]/I)$ (see also 
			\cite[Theorem~10.7]{Ma1986Commutative-ring-t}).
			Let $g, h \in L[y]/J$ be non-zero elements and assume that 
			$h / g \in (A/I)_{\mm}$.
			Thus it follows for the valuation $\omega$ defined in 
			Remark~\ref{rem:valuation} that
			$\omega(h / g) \geq 0$. 
			There exist integers $a, b$ such that $\omega(g) = a / b$
			and we can assume that $b > 0$. Thus we have
			$\omega(g^b /t^a) = 0$.
			Since $\omega(h) \geq \omega(g)$ we get 
			$\omega(h g^{b-1} /t^a) \geq 0$. Thus 
			$g^b/t^a$ and $h g^{b-1}/t^a$ both lie inside $A/I$.
			Using the fact that
			\[
				B/J = A/I \cap L[y] / J
				\subseteq K[y]/I
			\]
			we get
			\[
				\frac{h}{g} = \frac{h \cdot (g^{b-1} /t^a)}{g^b /t^a} \in 
				(B/J)_{\nn} \, .
			\]
			Thus we have $(A/I)_{\mm} \cap Q(L[y]/J) = (B/J)_{\nn}$.
			Note that the reasoning is similar to the proof of
			Lemma~\ref{lem:flat}.
			
			\item We claim that the complement of the image
			of $\Spec L[y] \to \Spec B$ is just the 
			point $\nn$. By Remark~\ref{rem:class}, the residue field of the 
			crucial maximal ideal $\mm \subseteq A$ is $\kk$
			and thus $\nn$ is a maximal ideal of $B$.
			Let $b \in \nn$. By Remark~\ref{rem:class} we have 
			$\mm = \rad(tA)$ and thus there exists an integer
			$q \ge 1$ such that $b^q \in tA$. Therefore $b^q / t \in A$. 
			Since $b^q / t \in L[y]$, we get 
			$b^q \in tB$. Thus we proved $\nn \subseteq \rad(t B)$.
			If $\pp \subseteq B$ is a prime ideal such that 
			$\pp L[y] = L[y]$, then we get $t \in \pp$ 
			(since $B_t = L[y]$). Thus we have
			$\nn \subseteq \rad(t B) \subseteq \pp$ 
			and by the maximality of $\nn$ we get $\nn = \pp$. 
			
			\item
			Now, we prove that $B/J$ is a maximal subring of $L[y]/J$. Let
			$C \subsetneq L[y]/J$ be a subring that lies over $B/J$.
			Using b), the fact that $J \subseteq \nn$ and 
			that $(B/J)_t = L[y] / J$, we get the
			following commutative diagram
			\[
				\xymatrix{
					& \Spec C \ar[rd]^-{\varphi} \\
					\Spec L[y] / J \ar[r]^-{\simeq} 
					\ar[ru]^-{\not\simeq}
					&  (\Spec B/J ) \setminus \{ \nn \} 
					\ar@{^(->}[r]_-{\textrm{open}} &
					\Spec B / J \, .
				}
			\]
			From this, one can easily deduce that $\varphi$ is surjective.	
			Let $\pp \in \Spec C$ with $\varphi(\pp) = \nn$.
			By a) and Lemma~\ref{lem:ExtendingMaxOfField},
			$(B/J)_{\nn}$ is a maximal subring of $Q(L[y]/J)$. Since $t \in \pp$,
			this implies $(B/J)_{\nn} = C_{\pp}$.
			Hence we have $B/J = C$ by
			\cite[Theorem~4.7]{Ma1986Commutative-ring-t}.
			\end{enumerate}
	\end{enumerate}
	From c) and from Lemma~\ref{lem:maxandcond} it follows that $B$
	is a maximal subring of $L[y]$. From b) it follows that $\nn = \mm \cap L[y]$
	is the crucial maximal ideal of $B$.
\end{proof}	     

\begin{remark}
	\label{rem:CondAndIntersection}
	If the conductor ideal of $A$ in $K[y]$ 
	is non-zero, then there exists $\alpha \in K^+$
	such that this ideal is $(y-\alpha)K[y]$.
	Now, if $L$ is a field and $L \subseteq K$ an algebraic field extension,
	then the conductor of $B = A \cap L[y]$ is the ideal $m_\alpha L[y]$
	where $m_\alpha \in L[y]$ is the minimal polynomial
	of $\alpha$ over $L$.
\end{remark}

In the future we will need the following consequence of the last two propositions.

\begin{corollary} 
\label{cor:BijectiveCorrespondenceAndLocalization}
We have a bijective correspondence
\[ 
	\varphi \colon
	\left\{	
	\begin{array}{c}
	    \textrm{maximal subrings of} \\
	    \textrm{$\kk[t, t^{-1}, y]$ that contain} \\ 
	    \textrm{$\kk[t, y]$}
	\end{array}
	\right\}
	\stackrel{1:1}{\longrightarrow}
	\left\{
	\begin{array}{c}
	    \textrm{maximal subrings of} \\
	    \textrm{$\kk(t)[y]$ that contain} \\
	    \textrm{$\kk(t)^+[y]$}
	\end{array}
	\right\}
\]
given by $\varphi(B) = B_S$ and 
$\varphi^{-1}(A) = A \cap \kk[t, t^{-1}, y]$, where $S$ denotes the multiplicative
subset $\kk[t] \setminus (t)$ of $\kk[t]$.
\end{corollary}

\begin{proof}
	Let $B \subsetneq \kk[t, t^{-1}, y]$ be a maximal subring that 
	contains $\kk[t, y]$. 
	By Lemma~\ref{lem:MaxSubringAndLocalization}, the
	localization $B_S$ is a maximal subring of 
	$\kk[t, t^{-1}, y]_S$, since $S = \kk[t] \setminus (t)$. 
	Moreover, we have
	\[
		B \subseteq B_S \cap 
		\kk[t, t^{-1},y] \subsetneq \kk[t, t^{-1}, y]
	\]
	and thus by the maximality of $B$ we get the equality
	$B = B_S \cap \kk[t, t^{-1},y]$. 
	This proves the injectivity of $\varphi$.
	
	Let $A \subsetneq \kk(t)[y]$ be a maximal subring that contains
	$\kk(t)^+[y]$. By Proposition~\ref{prop:existenceAbove} there exists 
	a maximal subring $K^+[y] \subseteq A' \subseteq K[y]$ such that
	$A' \cap \kk(t)[y] = A$. By Proposition~\ref{prop:interscting}, it follows
	that $A \cap \kk[t, t^{-1}, y]$ is a maximal subring of $\kk[t, t^{-1}, y]$.
	Clearly, $A \cap \kk[t, t^{-1}, y]$ contains $\kk[t, y]$.	Moreover,
	\[
		(A \cap \kk[t, t^{-1}, y])_S \subseteq A
	\]
	and by the maximality of $(A \cap \kk[t, t^{-1}, y])_S$ we get equality.
	This proves the surjectivity of $\varphi$.
\end{proof}

	     
\subsection{\texorpdfstring
	{Classification of the maximal subrings of $\kk[t, t^{-1}, y]$ 
	that contain $\kk[t, y]$}
	{Classification of the maximal subrings of k[t, t^{-1}, y] 
	that contain k[t, y]} }	     
\label{subsec:Class}
Throughout this subsection
$\KK$ denotes the algebraic closure of $\kk(t)$ inside the Hahn field $\kk[[t^\QQ]]$.
In this subsection we give a classification of all maximal subrings
of $\kk[t, t^{-1}, y]$ that contain $\kk[t, y]$.

Let $\alpha$ be in $\kk[[t^{\QQ}]]^+$. In this subsection we denote
\[
	A_{\alpha} = \kk[[t^{\QQ}]]^+ + (y-\alpha)\kk[[t^{\QQ}]][y] \, .
\] 
Thus $\alpha \mapsto A_{\alpha}$ is a bijective correspondence
between $\kk[[t^{\QQ}]]^+$ and the maximal subrings of $\kk[[t^{\QQ}]][y]$
that contain $\kk[[t^{\QQ}]]^+[y]$ by Corollary~\ref{cor:HahnClass} 
and Remark~\ref{rem:HahnClass}.

Let $(\QQ/\ZZ)^\ast$ be the group of group homomorphisms 
$\QQ / \ZZ \to \kk^\ast$.
There exists a natural action of this group on the Hahn field, given by the 
homomorphism
\begin{equation}
	\label{eq:OriginalAction}
	(\QQ/\ZZ)^\ast \longrightarrow \Gal(\kk[[t^\QQ]] / \kk((t))) \, , \quad
	\sigma \longmapsto \left( \sum_{s \in \QQ} a_s t^s \mapsto 
	\sum_{s \in \QQ} a_s \sigma(s)t^s \right) \, ,
\end{equation}
where $\Gal(\kk[[t^\QQ]] / \kk((t)))$ denotes the group of field automorphisms
of $\kk[[t^\QQ]]$ that fix the subfield $\kk((t))$ pointwise
(note that $\kk((t)) \subseteq \kk[[t^\QQ]]$ is  a Galois extension if and only
if the characteristic of $\kk$ is zero).
The action $\eqref{eq:OriginalAction}$ commutes with the valuation $\nu$ 
on $\kk[[t^\QQ]]$. In particular 
we have for all $\sigma \in (\QQ/\ZZ)^\ast$ and for all
$\alpha \in \kk[[t^\QQ]]^+$
\[
	A_\alpha \cap \kk[t, t^{-1}, y] = A_{\sigma(\alpha)} \cap \kk[t, t^{-1}, y] \, .
\]

The following result is the main theorem of this section.

\begin{theorem}
	\label{thm:Uniqueness}
	Let $\mathscr{S}$ be the set of $\alpha \in \kk[[t^\QQ]]^+$
	such that $\supp(\alpha)$ is contained in a strictly increasing sequence
	(see Definition~\ref{def:StrictlyIncreasingSequence}).
	Then we have a bijection
	\begin{align*}
		\Psi \colon 
		\mathscr{S} / (\QQ/\ZZ)^\ast & \longrightarrow
		\left\{
		    \begin{array}{c} 
			\textrm{maximal subrings of} \\ 
			\textrm{$\kk[t, t^{-1}, y]$ that contain $\kk[t, y]$}		    
		   \end{array}
		\right\}
		\, , \ \ \alpha \longmapsto A_\alpha \cap \kk[t, t^{-1}, y] \, .
	\end{align*}
	Moreover, $\Psi(\alpha)$ has non-zero conductor in $\kk[t, t^{-1}, y]$
	if and only if $\alpha \in K^+$ where $K$ denotes the algebraic closure
	of $\kk(t)$ inside the Hahn field $\kk[[t^\QQ]]$.
\end{theorem}


For the proof 
we need some preparation. First, we reformulate
the action of $(\QQ/\ZZ)^\ast$
on the Hahn field. Let $\kk(t^\QQ)$ be the subfield of the Hahn field generated
by the ground field $\kk$ and the elements $t^s$, $s \in \QQ$.
Then, $(\QQ/\ZZ)^\ast$ is isomorphic to the group
$\Gal(\kk(t^\QQ) / \kk(t))$ of field automorphisms of $\kk(t^\QQ)$ that fix
$\kk(t)$ pointwise. An isomorphism is given by
\[
	(\QQ/\ZZ)^\ast \longrightarrow \Gal(\kk(t^\QQ) / \kk(t)) \, , \quad
	\sigma \longmapsto \left( t^s \mapsto \sigma(s) t^s \right) \, ,
\]
and the homomorphism \eqref{eq:OriginalAction} identifies then 
under this isomorphism with 
\[
	\Gal(\kk(t^\QQ) / \kk(t)) \longrightarrow \Gal(\kk[[t^\QQ]] / \kk((t))) \, , 
	\quad
	\varphi \longmapsto \left( \sum_{s \in \QQ} a_s t^s \mapsto 
	\sum_{s \in \QQ} a_s \varphi(t^s) \right)
\]
(note that $\varphi(t^s)$ is a multiple of $t^s$ with some element of $\kk^\ast$).
For proving the injectivity of the map $\Psi$ in Theorem~\ref{thm:Uniqueness} 
we need two lemmas.

\begin{lemma}
\label{lem:wiggling} 
Let $q \in \QQ_{\geq 0}$ and let $\alpha$, $\alpha' \in \kk[[t^\QQ]]^+$.
Assume that we have decompositions
\[
	\alpha=\alpha_0+\alpha_1 \, , \quad \alpha'=\alpha_0+ct^q+\alpha_1'
	\quad \textrm{with} \quad 
	\textrm{$\alpha_0$, $\alpha_1$, $\alpha_1' \in \kk[[t^\QQ]]$} \, , \quad
	\textrm{$c \in \kk$}
\]
such that
\[
	\supp(\alpha_0)\subseteq [0, q] \, , \quad
	\supp(\alpha_1) \, , \ \supp(\alpha_1') \subseteq (q,\infty) \, , \quad
	\textrm{$\supp(\alpha_0)$ is finite} \, .
\] 
If $\nu(f(\alpha))=\nu(f(\alpha'))$ for all $f\in \kk(t)[y]$, then 
$\alpha_0+ct^q=\sigma(\alpha_0)$ for some 
$\sigma \in \Gal(\kk(t^\QQ) / \kk(t))$.
\end{lemma}

\begin{proof} 
Let $m_0 \in \kk(t)[y]$ be the minimal polynomial of $\alpha_0$ over $\kk(t)$. 
Note that $\alpha_0$ is algebraic over $\kk(t)$ since the support of $\alpha_0$ 
is a finite set.
Denote by $\alpha_0 = \beta_0, \ldots, \beta_r$ the different elements
of the set  
\[
	\{ \, \sigma(\alpha_0) \ | \ \sigma \in \Gal(\kk(t^\QQ)/\kk(t)) \, \} \, .
\]
As the field extension $\kk(t)Ê\subseteq \kk(t^\QQ)$ is normal,
there exist integers $k_0 > 0$ and $k_1, \ldots, k_r \geq 0$ such that
\[
	m_0 =(y-\beta_0)^{k_0}(y-\beta_1)^{k_1} \cdots 
	(y-\beta_r)^{k_r} \, ,
\]
see \cite[Theorem~3.20]{Mo1996Field-and-Galois-t}.
Assume towards a contradiction that 
$\alpha_0+ct^q\not =\beta_j$ for all $0\leq j \leq r$. Let $i$
be an integer with $1 \leq i \leq r$. 
Since $\nu(\alpha_0 - \beta_i) \leq q$, we get
\begin{eqnarray*}
	\nu(\alpha_1 + \alpha_0 - \beta_i) &=& 
	\nu(\alpha_0 - \beta_i) \\
	&=& \nu(ct^q + \alpha_0 - \beta_i) \\
	&=& \nu(\alpha'_1 + ct^q + \alpha_0 - \beta_i)
\end{eqnarray*}
where we used in the second and third equality the fact that 
$ct^q + \alpha_0  \neq \beta_i$. Since 
$\alpha_0 = \beta_0 \neq \alpha_0 + ct^q$, the constant $c$ is non-zero.
Thus we have $\nu(\alpha_1) > \nu(\alpha'_1 + ct^q)$. In summary we get
\begin{eqnarray*}
\nu(m_0(\alpha)) &=& k_0 \nu(\alpha_1) + 
			\sum_{i \neq 0} k_i \nu(\alpha_1+\alpha_0-\beta_i) \\
		    &>& k_0 \nu(\alpha'_1+ ct^q) + \sum_{i \neq 0} 
		    	k_i \nu(\alpha_1'+ct^q+\alpha_0-\beta_i) \ \ = \ \ 
			\nu(m_0(\alpha'))
\end{eqnarray*}
and thus we arrive at a contradiction.
\end{proof}

\begin{lemma} 
\label{lem:ForInjectivity}
Let $\alpha$, $\alpha' \in \kk[[t^{\QQ}]]^+$ and assume that
$\supp(\alpha)$, $\supp(\alpha')$ are contained in strictly increasing sequences
(see Definition~\ref{def:StrictlyIncreasingSequence}).
Then $\nu(f(\alpha))=\nu(f(\alpha'))$ for all $f\in \kk(t)[y]$ if and only if there 
exists $\sigma \in \Gal(\kk(t^\QQ) / \kk(t))$ 
such that $\alpha'=\sigma(\alpha)$.
\end{lemma}

\begin{proof} 
Assume that $\nu(f(\alpha))=\nu(f(\alpha'))$ for all $f\in \kk(t)[y]$.
By assumption, there exists a strictly increasing sequence
$0 < s_1 < s_2 < \ldots$ in $\QQ$ such that
\[
	\alpha = a_0 + \sum_{j = 1}^{\infty} a_{s_j} t^{s_j} \quad \textrm{and} \quad
	\alpha' = a'_0 + \sum_{j = 1}^{\infty} a'_{s_j} t^{s_j} \, .
\]
For $i \geq 1$ let
\[
	\alpha_i = a_0 + \sum_{j = 1}^{i-1} a_{s_j} t^{s_j} \quad \textrm{and} \quad
	\alpha'_i = a'_0 + \sum_{j = 1}^{i-1} a'_{s_j} t^{s_j} \, .
\]
We define inductively
$\sigma_1, \sigma_2, \ldots \in \Gal(\kk(t^\QQ) / \kk(t))$ such
that $\sigma_i(\alpha_i) = \alpha'_i$. 

Since 
$\nu(\alpha-a_0) = \nu(\alpha' -a_0)$ by assumption, 
it follows that $a'_0 = a_0$.
Hence $\sigma_1 = \id$ satisfies $\sigma_1(\alpha_1) = \alpha'_1$.
Assume that $\sigma_i \in \Gal(\kk(t^\QQ) / \kk(t))$ with 
$\sigma_i(\alpha_i) = \alpha'_i$ is already constructed. 
For all $f \in \kk(t)[y]$ we have
\[
	\nu(f(\sigma_i(\alpha))) = \nu(\sigma_i(f(\alpha))) = 
	\nu(f(\alpha))= \nu(f(\alpha')) \, .
\]
Since $\alpha'_i = \sigma_i(\alpha_i)$, Lemma~\ref{lem:wiggling} implies that
there exists $\varphi \in \Gal(\kk(t^\QQ) / \kk(t))$ such that 
$\alpha'_{i+1} = \varphi(\sigma_i(\alpha_{i+1}))$. Thus we can define
$\sigma_{i+1} = \varphi \circ \sigma_i$.

By construction, $\sigma_{i+1}$ and $\sigma_i$ coincide on the field
\[
 	K_i = \kk \left( \{ \, t^s \ | \ s \in \supp(\alpha_i) \, \} \right) \, .
\]
Thus we get a well defined automorphism of the field
$\bigcup_{i=0}^{\infty} K_i$ that restricts to $\sigma_i$ on $K_i$. 
By the normality of the extension $\kk(t) \subseteq \kk(t^\QQ)$ we 
we can extend this automorphism
to an automorphism $\sigma$ of $\kk(t^\QQ)$ and we have 
$\sigma(\alpha) = \alpha'$ (see \cite[Theorem~3.20]{Mo1996Field-and-Galois-t}).

The converse of the statement is clear.
\end{proof}
	     
\begin{proof}[Proof of Theorem~\ref{thm:Uniqueness}]
	Consider the bijections
	\begin{align*}
		\Xi_1 \colon \KK^+ & \ \longrightarrow \
		\left\{
		    \begin{array}{c} 
			\textrm{maximal subrings of $\KK[y]$ with} \\ 
			\textrm{non-zero conductor that contain $\KK^+[y]$}
		   \end{array}
		\right\} \\
		\Xi_2 \colon \mathscr{S} \setminus \KK^+ & \ \longrightarrow \
		\left\{
		    \begin{array}{c} 
			\textrm{maximal subrings of $\KK[y]$ with} \\ 
			\textrm{zero conductor that contain $\KK^+[y]$}
		    \end{array}
		\right\} \, .
	\end{align*}	
	of Proposition~\ref{prop:AlgClorClass}. For $\alpha \in \KK^+$
	and $\beta \in \mathscr{S} \setminus \KK^+$ we have
	\[
		\Xi_1(\alpha) \cap \kk[t, t^{-1}, y] = \Psi(\alpha) \quad
		\textrm{and} \quad
		\Xi_2(\beta) \cap \kk[t, t^{-1}, y] = \Psi(\beta) \, .
	\]
	Using Proposition~\ref{prop:interscting} and 
	Remark~\ref{rem:CondAndIntersection} we see that
	$\Xi_1(\alpha) \cap \kk[t, t^{-1}, y]$ is a maximal subring of 
	$\kk[t, t^{-1}, y]$ with non-zero
	conductor and $\Xi_2(\beta) \cap \kk[t, t^{-1}, y]$ is a maximal subring of 
	$\kk[t, t^{-1}, y]$
	with zero conductor. Thus $\Psi$ is a well-defined map.
	Using Proposition~\ref{prop:existenceAbove}, we see that 
	$\Psi$ is surjective.
	
	For proving the injectivity, let $\alpha_1, \alpha_2 \in \mathscr{S}$ 
	such that the rings $A_{\alpha_1} \cap \kk[t, t^{-1}, y]$, 
	$A_{\alpha_2} \cap \kk[t, t^{-1}, y]$ are the same subsets of $\kk[t, t^{-1}, y]$.
	For $i=1, 2$, $A_{\alpha_i} \cap \kk(t)[y]$
	is a maximal subring of $\kk(t)[y]$, see Proposition~\ref{prop:interscting}.
	By Corollary~\ref{cor:BijectiveCorrespondenceAndLocalization}, we get 
	the equality
	\[
		A_{\alpha_1} \cap \kk(t)[y] = A_{\alpha_2} \cap \kk(t)[y] \, .
	\]
	Let $B = A_{\alpha_1} \cap \kk(t)[y] = A_{\alpha_2} \cap \kk(t)[y]$. 
	Let $\nn$ be the crucial maximal ideal of $B$
	and let $J$ be the conductor ideal of $B$ in $\kk(t)[y]$.
	With Remark~\ref{rem:valuation} we get for $i=1, 2$ 
	\[
		B = \{ \, f \in \kk(t)[y] \ | \ \omega_i(\bar{f}) \geq 0 \, \} \quad \textrm{and}
		\quad \nn = \{ \, f \in \kk(t)[y] \ | \ \omega_i(\bar{f}) > 0 \, \} \, ,
	\]
	where $\bar{f}$ denotes the residue class modulo $J$
	and $\omega_i$ denotes the valuation
	\[
		\omega_i \colon Q(\kk(t)[y]/J) \longrightarrow \QQ \, , \quad 
		g \longmapsto \nu(g(\alpha_i)) \, .		
	\]
	By~\cite[Proposition 3.3]{FeOl1970Homomorphisms-mini},
	$(B/J)_{\nn}$ is a one-dimensional valuation ring of the field $Q(\kk(t)[y]/J)$
	and therefore it is a maximal subring of $Q(\kk(t)[y]/J)$, 
	see Lemma~\ref{lem:ExtendingMaxOfField}. The description above of $B$
	and $\nn$ implies that $(B/J)_{\nn}$ is the valuation ring with respect to
	$\omega_1$ and with respect to $\omega_2$. Therefore, the valuations 
	$\omega_1$, $\omega_2$ are the same up to an order preserving
	isomorphism of $(\QQ, +, <)$. However, since 
	$\omega_1(t) = 1 = \omega_2(t)$, these
	valuations must then be the same. Thus by Lemma~\ref{lem:ForInjectivity} 
	there exists $\sigma \in \Gal(\kk(t^\QQ) / \kk(t))$ such that
	$\alpha_1 = \sigma(\alpha_2)$. This proves the injectivity of $\Psi$.
\end{proof}     

\section{\texorpdfstring
	{Classification of the maximal $\kk$-subalgebras of $\kk[t, t^{-1}, y]$}
	{Classification of the maximal k-subalgebras of k[t, t^{-1}, y]} }
\label{classOfk[tt-1y].sec}	      
	     
The goal of this section is to classify all maximal $\kk$-subalgebras of 
$\kk[t, t^{-1}, y]$. In fact, we reduce this problem in this section 
to another classification result, which we will solve then in the next section. 
	      
\begin{proposition}
	\label{prop:classOfk[tt-1y]}
	Let $A \subseteq \kk[t, t^{-1}, y]$ be an extending 
	maximal $\kk$-subalgebra. Then, exactly one of the following cases occur:
	\begin{enumerate}[i)]
	\item There exists and automorphism $\sigma$ of $\kk[t, t^{-1}, y]$ 
		such that $\sigma(A)$ contains $\kk[t, y]$;
	\item $A$ contains $\kk[t, t^{-1}]$.
	\end{enumerate}
\end{proposition}

\begin{proof}[Proof of Proposition~\ref{prop:classOfk[tt-1y]}]
	Note that $A$ satisfies the property $P_2$
	in $\kk[t, t^{-1}, y]$, see Lemma~\ref{lem:PropertyP_2ForExtMax}. 
	Since $t \cdot t^{-1}  = 1 \in A$,
	it follows that either $t \in A$ or $t^{-1} \in A$.
	Assume that we are not in case ii), i.e. assume 
	that $\kk[t, t^{-1}]$ is not contained in $A$. By applying
	an appropriate automorphism of $\kk[t, t^{-1}, y]$, 
	we can assume that $t \in A$ and hence $t^{-1} \not \in A$. Therefore we get 
	$A_t = A[t^{-1}] = \kk[t, t^{-1}, y]$, since $A$ is maximal.
	This implies that there exists an integer $k \geq 0$ such that
	$t^k y \in A$. Thus the $\kk[t, t^{-1}]$-automorphism
	\[
		\sigma \colon \kk[t, t^{-1}, y] \longrightarrow \kk[t, t^{-1}, y] \, , \quad
		y \longmapsto t^{-k} y 
	\]
	satisfies $y \in \sigma(A)$. Hence we get
	$\sigma(A) \supseteq \kk[t, y]$ and therefore we are in case i).
\end{proof}

The extending maximal $\kk$-subalgebras in case i) of 
Proposition~\ref{prop:classOfk[tt-1y]} 
are then described by Theorem~\ref{thm:Uniqueness}.
Thus we are left with the description of the extending maximal $\kk$-subalgebras
in case ii). In fact, they can be characterized in the following way:
	      
\begin{proposition}
	\label{prop:ExtMaxContainingtandtinv}
	There is a bijection
	\[
		\Phi \colon
		\left\{
			\begin{array}{c}
				\textrm{extending maximal} \\ 
				\textrm{$\kk$-subalgebras of $\kk[t, t^{-1}, y]$} \\
				\textrm{that contain $\kk[t, t^{-1}]$}
			\end{array}
		\right\}
		\longrightarrow
		\left\{
			\begin{array}{c}
				\textrm{extending maximal} \\
				\textrm{$\kk$-subalgebras of $\kk[t, y]$ that} \\
				\textrm{contain $\kk[t]$ and $t$ lies not in} \\
				\textrm{the crucial maximal ideal}
			\end{array}
		\right\}
	\]
	given by $\Phi(A) = A \cap \kk[t, y]$.
\end{proposition}	      

\begin{proof}
	Let $A$ be an extending maximal $\kk$-subalgebra of 
	$\kk[t, t^{-1}, y]$ that contains $\kk[t, t^{-1}]$. 
	By Lemma~\ref{lem:ResidueField}, there exists 
	$\lambda \in \kk^\ast$ such that $t-\lambda$ lies in the crucial
	maximal ideal $\mm$ of $A$. Thus 
	$A_{t-\lambda} = \kk[t, t^{-1}, y]_{t-\lambda}$.
	Hence there exists $k \geq 1$ such that $(t-\lambda)^k y \in A$ and thus 
	we get
	\[
		\kk[t, t^{-1}, (t-\lambda)^k y] \subseteq A \subsetneq \kk[t, t^{-1}, y] \, .
	\]
	This implies
	\[
		\kk[t, (t-\lambda)^k y] \subseteq A \cap \kk[t, y] \subsetneq \kk[t, y] \, .
	\]
	We claim, that $A \cap \kk[t, y]$ is a maximal subring of $\kk[t, y]$. 
	Let therefore $A \cap \kk[t, y] \subseteq B \subseteq \kk[t, y]$ 
	be an intermediate ring. Thus we get
	\[
		B = B_{t -\lambda} \cap B_t \supseteq \kk[t, y]Ê\cap B_t \supseteq B \, .
	\]
	One can check that 
	$A = (A \cap \kk[t, y])_t$. Since $A$ is maximal in $\kk[t, t^{-1}, y]$, we get 
	either $B_t = A$ or $B_t = \kk[t, t^{-1}, y]$ and the claim follows.
	This proves that $\Phi$ is well-defined and injective.
	
	Let $A'$ be an extending maximal $\kk$-subalgebra of $\kk[t, y]$ that contains
	$\kk[t]$ and the crucial maximal ideal does not contain $t$.
	By Lemma~\ref{lem:MaxSubringAndLocalization} it follows that $A'_t$
	is an extending maximal $\kk$-subalgebra of $\kk[t, t^{-1}, y]$ that contains
	$\kk[t, t^{-1}]$. Moreover, we have $A'_t \cap \kk[t, y] = A$
	by the maximality of $A$ in $\kk[t, y]$. This proves the surjectivity
	of $\Phi$.
\end{proof}
	     
After this proposition, one is now reduced to the problem of the description
of all maximal $\kk$-subalgebras of $\kk[t, y]$ that contain $\kk[t]$.
	     
\section{\texorpdfstring
	{Classification of the maximal $\kk$-subalgebras of $\kk[t, y]$ 
	that contain $\kk[t]$}
	{Classification of the maximal k-subalgebras of k[t, y] 
	that contain k[t]}}
\label{ClassWithCoord.sec}
	      	      
Let $\MMM$ be the set of extending maximal $\kk$-subalgebras of 
$\kk[t, y]$ that contain $\kk[t]$. The goal of this section is to describe the set
$\MMM$ with the aid of the classification result Theorem~\ref{thm:Uniqueness}.
For this we introduce a subset $\NNN$ of the maximal $\kk$-subalgebras
of $\kk[t, y, y^{-1}]$ that contain $\kk[t, y^{-1}]$.
     
\begin{remark}
	\label{rem:TothesetN}
	If $A$ is an extending maximal $\kk$-subalgebra of $\kk[t, y, y^{-1}]$
	that contains $\kk[t, y^{-1}]$, then the residue field of the crucial maximal
	ideal is  isomorphic to $\kk$, by Remark~\ref{rem:class} and 
	Theorem~\ref{thm:Uniqueness}. Hence there exists a unique 
	$\lambda \in \kk$ such that $t-\lambda$ lies in 
	the crucial maximal ideal of $A$.
\end{remark}     

Define $\NNN$ to be the set of extending maximal $\kk$-subalgebras $A$ of 
$\kk[t, y, y^{-1}]$ that contain $\kk[t, y^{-1}]$ and such that
\begin{equation}
	\label{eq:isomor}
	 A \longrightarrow \kk[t, y, y^{-1}] / (t-\lambda)
\end{equation}
is surjective where $\lambda$ denotes the unique
element in $\kk$ such that the crucial maximal ideal contains $t - \lambda$
(see Remark~\ref{rem:TothesetN}).
Now, we can formulate the main result of this section.

\begin{theorem}
	\label{thm:MandN}
	The map $\Theta \colon \NNN \to \MMM$, $A \mapsto A \cap \kk[t, y]$
	is bijective.
\end{theorem}

\begin{remark}
\label{rem:ClassMaxSubWithCoordinate}
As we classified already all maximal subrings of $\kk[t, y, y^{-1}]$ that
contain $\kk[t, y^{-1}]$ (see Theorem~\ref{thm:Uniqueness}), 
Theorem~\ref{thm:MandN}
gives us a description of all extending maximal $\kk$-subalgebras 
of $\kk[t, y]$ that contain a coordinate
of $\kk[t, y]$ (up to automorphisms of $\kk[t, y]$). 
\end{remark}

\begin{remark}
\label{rem:ReductiontoM0andN0}
Lemma~\ref{lem:ResidueField} implies the following: If $A$
is an extending maximal $\kk$-subalgebra of $\kk[t, y]$ 
which contains $\kk[t]$, then there exists a unique
$\lambda \in \kk$ such that $t-\lambda$ lies in the crucial maximal
ideal of $A$. Thus $\MMM$ is the disjoint union
of the sets 
\[	
 	\MMM_\lambda = \{ \, A \in \MMM \ | \ 
	\textrm{$t-\lambda$ lies in the crucial maximal ideal of $A$} \, \} \, , \quad
	\lambda \in \kk \, .
\]
By Remark~\ref{rem:TothesetN}, $\NNN$ is the disjoint union of the sets
\[
	\NNN_\lambda = \{ \, A \in \NNN \ | \ 
	\textrm{$t-\lambda$ lies in the crucial maximal ideal of $A$} \, \} \, , \quad
	\lambda \in \kk \, .	
\]
Note that we have canonical bijections
\[
	\MMM_0 \longmapsto \MMM_{\lambda} \, , 
	\quad A \longmapsto \sigma_\lambda(A)
	\quad \textrm{and} \quad
	\NNN_0 \longrightarrow 
	\NNN_{\lambda} \, , \quad A \longmapsto \sigma_\lambda(A)
\]
where $\sigma_{\lambda}$ is the automorphism of $\kk[t, y, y^{-1}]$ given by
$\sigma_{\lambda}(t) = t-\lambda$ and $\sigma_{\lambda}(y) = y$. 
Using the fact that for all $A \in \MMM_0$ we have
\[
	\sigma_\lambda(A) \cap \kk[t, y] = \sigma_\lambda(A \cap \kk[t, y]) \, ,
\]
one is reduced for the proof of Theorem~\ref{thm:MandN} 
to proving the following proposition.
\end{remark}

\begin{proposition}
	\label{prop:MandN}
	The map $\NNN_0 \to \MMM_0$, $A \mapsto A \cap \kk[t, y]$
	is bijective.
\end{proposition}


For the proof of Proposition~\ref{prop:MandN} we need several
(technical) lemmas.

\begin{lemma}
	\label{helping.lem}
	Let $\kk[t] \subseteq Q \subsetneq \kk[t, y]$ be an intermediate ring
	that satisfies the $P_2$ property in $\kk[t, y]$ and assume that
	\[
		Q \longrightarrow \kk[t, y] / t \kk[t, y]
	\]
	is surjective. If $\pp \subseteq Q$ is an ideal that contains $t$ and
	that does not contain $t \kk[t, y] \cap Q$, 
	then there exists $h \in Q \setminus \pp$
	such that $y^{-1} \in Q_h$.
\end{lemma}

\begin{proof}
	By assumption, there exists $g \in \kk[t, y] \setminus y \kk[t, y]$ 
	and $n \geq 0$ such that 
	\begin{equation}
		\label{eq4}
		ty^n g \in Q \setminus \pp \, . 
	\end{equation}
	Let $g_0 \in \kk[t]$, $g' \in \kk[t, y] \setminus y \kk[t, y]$ 
	and $r \geq 1$ such that $g-g_0 = y^r g'$.
	If $n = 0$, we get 
	\[
		ty^r g' = t g - t g_0 \in Q \setminus \pp \, .
	\]
	Thus we can and will assume that $n \geq 1$. Now, choose
	$g \in \kk[t, y] \setminus y \kk[t, y]$ of minimal $y$-degree such that 
	$\eqref{eq4}$ is satisfied for some $n \geq 1$.
	We claim that $g \in Q$. Otherwise, $\deg_y(g) > 0$ and 
	$ty^n \in Q$, since $Q$ satisfies the $P_2$ property in $\kk[t, y]$.
	In fact, since $g$ is of minimal 
	$y$-degree, we get $ty^n \in \pp$.
	Thus we get a contradiction to the fact that
	\[
		t y^{n+r}g' = t y^n g - t y^n g_0
		\in Q \setminus \pp \quad \textrm{and} \quad
		\deg_y(g') < \deg_y(g) \, .
	\]
	
	Let $h = ty^n g \in Q \setminus \pp$.
	Since $t, g \in Q$, it follows that $y^{-n} = tg / h \in Q_h$.
	Since $Q$ satisfies the property $P_2$ in $\kk[t, y]$, 
	the localization $Q_h$ satisfies the property $P_2$ in $\kk[t, y]_h = 
	\kk[t, t^{-1}, y, y^{-1}]_g$. Hence, we get $y^{-1} \in Q_h$.
\end{proof}

\begin{lemma}
	\label{lem:openimmersion}
	Let $A \in \NNN_0$ and let $\mm$ be the crucial maximal ideal of $A$.
	Then the inclusion $A \cap \kk[t, y] \subseteq \kk[t, y]$ defines
	an open immersion
	\[
		\varphi \colon \AA^2_{\kk} \longrightarrow \Spec A \cap \kk[t, y]
	\]
	on spectra and the complement of the image of $\varphi$ consists 
	only of the maximal ideal $\mm \cap \kk[t, y]$ of $A \cap \kk[t, y]$.
\end{lemma}

\begin{remark}
	\label{rem:tolemOpenimmersion}
	The proof will show the following: 
	\begin{enumerate}[a)]
	\item the maximal ideal $\mm \cap \kk[t, y]$ of 
	$A \cap \kk[t, y]$ contains $t$ and
	does not contain $t \kk[t, y] \cap A \cap \kk[t, y]$ (see iii) in the proof);
	
	\item the homomorphism 
	$A \cap \kk[t, y] \to \kk[t, y] / t \kk[t, y]$ is surjective
	(see i) in the proof).
	\end{enumerate}
\end{remark}

\begin{proof}[Proof of Lemma~\ref{lem:openimmersion}]
	Let $A' = A \cap \kk[t, y]$ and let $\mm' = \mm \cap \kk[t, y]$. 
	Due to Remark~\ref{rem:TothesetN}, 
	the residue field $A / \mm$ is isomorphic to $\kk$.
	Hence, $\mm'$ is a maximal ideal of $A'$.
	We divide the proof in several steps.
	\begin{enumerate}[i)]
		\item We claim that $\varphi$ induces a closed immersion
			$\{ 0 \} \times \AA_{\kk}^1 \to V_{\Spec(A')}(t)$.
			Due to the surjection~\eqref{eq:isomor},
			there exists $f \in \kk[t, y, y^{-1}]$, $a \in A$ such that
			$y = a + tf$.
			Let $f = f^+ + f^-$ where $f^+ \in \kk[t, y]$ and $\deg_y(f^-) < 0$.
			We have $a + tf^- \in A$ and $tf^+ \in t\kk[t, y]$. Thus we get
			\[
				a + tf^- = y -tf^+ \in A \cap \kk[t, y] = A' \, .
			\]
			This implies that
			\[
				A' / t A' \longrightarrow \kk[t, y] / t\kk[t, y] = \kk[y]
			\]
			is surjective, which implies the claim.
		\item	We claim that $\varphi$ induces an isomorphism
			 $\AA_{\kk}^\ast \times \AA_{\kk}^1 \simeq
			 \Spec (A') \setminus V_{\Spec(A')}(t)$.
			 Since $t \in \mm$, we have 
			$A_t = \kk[t, t^{-1}, y, y^{-1}]$ and 
			thus $t^k y \in A$ for some integer $k$.
			This implies $t^k y \in A'$ and thus $A'_t = \kk[t, t^{-1}, y]$.
		\item We claim that
			$\Spec(A') \setminus \varphi(\AA^2_{\kk}) = \{ \mm' \}$.
			Using i) and ii) this is equivalent to show that 
			$\mm'$ is the only prime ideal of $A'$ that contains $t$
			and does not contain $t \kk[t, y] \cap A'$. 
			
			Since $\mm$ contains $t$ it follows that $\mm'$
			contains $t$. Since there exists no prime ideal of 
			$\kk[t, y, y^{-1}]$ that lies over $\mm$, the surjection~
			\eqref{eq:isomor} implies that 
			$\mm$ does not contain $t \kk[t, y, y^{-1}] \cap A$. 
			Hence there exists 
			$f \in \kk[t, y, y^{-1}]$ such that $t f \in A \setminus \mm$.
			Since $t \kk[t, y^{-1}] \subseteq \mm$, we can even assume
			that $f \in \kk[t, y]$. Hence $t f \in A' \setminus \mm'$ and
			therefore $\mm'$ does not contain $t \kk[t, y] \cap A'$.
			
			As $A \subseteq \kk[t, y, y^{-1}]$ induces an isomorphism
			$\AA_{\kk}^1 \times \AA_{\kk}^\ast \simeq 
			\Spec(A) \setminus \{ \mm \}$ and since $t \in \mm$, we have 
			\[
				\rad(tA) = t \kk[t, y, y^{-1}] \cap A \cap \mm \, .
			\]
			Intersecting with $\kk[t, y]$ yields
			\[
				\rad(tA') = t \kk[t, y] \cap A' \cap \mm'.
			\]
			Thus every prime ideal of $A'$
			that contains $t$ and does not contain 
			$t \kk[t, y] \cap A'$ must be equal to $\mm'$
			(note that $\mm'$ is a maximal ideal of $A'$).
		\item We claim that $\varphi$ is an open immersion.
			According to Theorem~\ref{thm:Uniqueness} and 
			Remark~\ref{rem:valuation}, there exists
			$\alpha \in \kk[[(y^{-1})^{\QQ}]]^+$ such that
			\[
				A = \{ \, f \in \kk[y^{-1}, y, t] \ | \ \nu(f(\alpha)) \geq 0 \, \} \, .
			\]
			Note that $y^{-1}$ corresponds 
			to the $t$ in Theorem~\ref{thm:Uniqueness}
			and $t$ corresponds to the $y$ in Theorem~\ref{thm:Uniqueness}.
			In particular we have $\nu(y^{-1}) = 1$.
			If $\alpha = 0$, then $A = \kk[y^{-1}] + t \kk[y^{-1}, y, t]$ 
			and thus \eqref{eq:isomor}
			is not surjective. Hence $\alpha \neq 0$.
			Let $\nu(\alpha) = a/b$ for integers $a \geq 0$, $b > 0$ and 
			let $\lambda \in \kk^\ast$ be the coefficient
			of $y^{-a/b}$ of $\alpha$. 
			There exists $k \geq 1$ such that
			\begin{equation}
				\label{eq20}
				y (\lambda^b -t^b y^a)^k \in A' \, .
			\end{equation}
			Indeed, $\nu(\lambda^b -\alpha^b y^a) > 0$, since
			$\alpha$ is equal to $\lambda (y^{-1})^{a/b}$ plus 
			higher oder terms in $y^{-1}$. 
			Hence, there exists $k \geq 1$ such that
			\[
				\nu(y (\lambda^b -\alpha^b y^a)^k) = 
				-1 + k \nu(\lambda^b -\alpha^b y^a) \geq 0 \, ,
			\]
			which yields \eqref{eq20}. 
			
			As $A$ satisfies the property $P_2$
			in $\kk[y^{-1}, y, t]$ 
			(see Lemma~\ref{lem:PropertyP_2ForExtMax}), it follows that 
			$A'$ satisfies the property $P_2$ in $\kk[y, t]$. Since $y \notin A'$
			we get thus $\lambda^b-t^b y^a \in A'$ by \eqref{eq20}. 
			Again by \eqref{eq20} we have 
			$y \in A'_{\lambda^b-t^b y^a}$, which implies 
			\[
				A'_{\lambda^b-t^b y^a} = \kk[t, y]_{\lambda^b-t^b y^a} \, .
			\]
			As the zero set of $\lambda^b-t^b y^a$ and of $t$ in 
			$\AA^2_{\kk} = \Spec \kk[t, y]$ are disjoint, it follows with ii) that
			$\varphi \colon \AA^2_{\kk} \to \Spec(A')$ is locally an 
			open immersion. However, i) and ii) imply that $\varphi$ is
			injective and thus $\varphi$ is an open immersion.
	\end{enumerate}
\end{proof}

\begin{lemma}
	\label{maxsub.lem}
	Let $A \in \NNN_0$ and let $\mm$
	be the crucial maximal ideal of $A$. 
	Moreover, we denote $A' = A \cap \kk[t, y]$ and 
	$\mm' = \mm \cap \kk[t, y]$.
	Then the following holds:
	\begin{enumerate}[a)]
		\item $A'$ is a maximal subring of $\kk[t, y]$;
		\item $\mm'$ is the crucial maximal ideal of $A'$;
		\item For all $h \in A' \setminus \mm'$
		such that $y^{-1} \in A'_h$ we have 
		\[
			A = A'_h \cap \kk[t, y, y^{-1}]
			\quad \textrm{and} \quad A'_h = A_h \, .
		\]
		Moreover, there exist $h \in A' \setminus \mm'$ with $y^{-1} \in A'_h$.
	\end{enumerate}
\end{lemma}


\begin{proof}[Proof of Lemma~\ref{maxsub.lem}]
	As $A$ satisfies the $P_2$ property in $\kk[t, y, y^{-1}]$,
	$A'$ satisfies the $P_2$ property in $\kk[t, y]$.
	By Remark~\ref{rem:tolemOpenimmersion}, $\mm'$ contains $t$ and
	does not contain $t \kk[t, y] \cap A'$. Moreover, the homomorphism
	$A' \to \kk[t, y] / t \kk[t, y]$
	is surjective according to Remark~\ref{rem:tolemOpenimmersion}.
	Let $h \in A' \setminus \mm'$ such that $y^{-1} \in A'_h$
	(by Lemma~\ref{helping.lem} there exists such an $h$).
	We claim that
	\begin{equation}
		\label{eq201}
		A'_h = A_h \, .
	\end{equation}
	Indeed, if $a = a^+ + a^- \in A$ and $a^+ \in \kk[t, y]$, $\deg_y(a^-) <0$,
	then we get 
	\[
		a^+ = a-a^-\in A \cap \kk[t, y] = A' \, .
	\] 
	However, $a^- \in \kk[t, y^{-1}] \subseteq A'_h$ and thus 
	$a = a^+ + a^- \in A'_h$, which
	implies the claim. Using Lemma~\ref{lem:MaxSubringAndLocalization}
	and the fact that $h \in A' \setminus \mm'$, 
	the claim implies that
	\[
		A'_h \subsetneq \kk[t, y, y^{-1}]_h = \kk[t, y]_h
	\]
	is an extending maximal subring. Now, let $A' \subseteq B \subsetneq \kk[t, y]$
	be an intermediate ring. Since $\varphi \colon \AA^2_{\kk} \to \Spec(A')$
	is an open immersion and 
	$\Spec(A') \setminus \varphi(\AA^2_{\kk}) = \{ \mm' \}$ 
	(see Lemma~\ref{lem:openimmersion}), it follows that $\mm'$
	lies in the image of the morphism $\Spec(B) \to \Spec(A')$. 
	Hence, there exists a prime ideal in $B$ that lies over $\mm'$.
	Since $h \in A' \setminus \mm'$, it follows that there exists a prime ideal
	of $B_h$ that lies over $\mm' A'_h$. In particular, $B_h \neq \kk[t, y]_h$. 
	By the maximality of $A'_h$ in $\kk[t, y]_h$ we get
	$A'_h = B_h$. Thus
	\[
		B \subseteq B_h \cap \kk[t, y] =
		A'_h \cap \kk[t, y] = A_h \cap \kk[t, y] = A' 	
	\]
	where the last equality follows from the fact that 
	\begin{equation}
		\label{eq21}
		A_h \cap \kk[t, y, y^{-1}] = A
	\end{equation}
	(note that $y \not\in A_h$, since otherwise 
	$y h^k \in A \cap \kk[t, y] = A'$ for a certain integer $k$
	and thus $y \in A'_h$, contradicting the maximality of $A'_h$ in $\kk[t, y]_h$). 
	This proves the maximality of $A'$ in $\kk[t, y]$, which is a).
	Equations~\eqref{eq201} and~\eqref{eq21} say, that c) is satisfied. 
	Statement b) is a consequence of statement a) and 
	Lemma~\ref{lem:openimmersion}.
\end{proof}

\begin{proof}[Proof of Proposition~\ref{prop:MandN}]
	From Lemma~\ref{maxsub.lem}~a),b)
	it follows that $\NNN_0 \to \MMM_0$ is well-defined. 
	From Lemma~\ref{maxsub.lem}~c)
	it follows that $\NNN_0 \to \MMM_0$ is injective. 
	
	Now, we prove the surjectivity. 
	Let $Q \in \MMM_0$. We have the following inclusion
	\begin{equation}
		\label{eq23}
		Q / t \kk[t, y] \cap QÊ\subseteq \kk[t, y] / t \kk[t, y] = \kk[y] \, .
	\end{equation}
	On spectra, this map yields an open immersion, since
	$\Spec \kk[t, y] \to \Spec Q $ is an open immersion. Hence,
	\eqref{eq23} is a finite ring extension, and thus \eqref{eq23} 
	must be an equality.	
	This implies that the crucial maximal ideal $\pp$ of $Q$ 
	does not contain $t \kk[t, y] \cap Q$ (note that 
	$\Spec Q \setminus \Spec \kk[t, y] = \{ \pp \}$).
	By assumption, $t \in \pp$. Moreover, $Q$ satisfies the $P_2$ property
	in $\kk[t, y]$ by Lemma~\ref{lem:PropertyP_2ForExtMax}.
	By Lemma~\ref{helping.lem} there
	exists $h \in Q \setminus \pp$ such that $y^{-1} \in Q_h$.
	Thus, Lemma~\ref{lem:MaxSubringAndLocalization} implies that
	\[
		Q_h \subsetneq \kk[t, y]_h = \kk[t, y, y^{-1}]_h
	\]
	is an extending maximal subring. Since $y^{-1} \in Q_h$, the ring
	\[
		Q' = Q_h \cap \kk[t, y, y^{-1}] \subsetneq \kk[t, y, y^{-1}]
	\]
	contains $\kk[t, y^{-1}]$. Now, we divide the proof in several steps.
	
	\begin{enumerate}[i)]
	\item We claim that $Q'$ is a maximal subring
	of $\kk[t, y, y^{-1}]$. Therefore, take an intermediate ring 
	$Q' \subseteq B \subsetneq \kk[t, y, y^{-1}]$. By the maximality of $Q$
	in $\kk[t, y]$ we get $Q = B \cap \kk[t, y]$ and hence
	\[
		Q_h = (B \cap \kk[t, y])_h \, .
	\]
	If $y$ would be in $B_h$, then $y$ would be in 
	$(B \cap \kk[t, y])_h = Q_h$, a contradiction to the fact that
	$Q_h \neq \kk[t, y, y^{-1}]_h$. Hence we have 
	$B_h \neq \kk[t, y, y^{-1}]_h$. The maximality of $Q_h$
	in $\kk[t, y, y^{-1}]_h$ implies that $B_h = Q_h$.
	Hence, we have
	\[
		B \subseteq B_h \cap \kk[t, y, y^{-1}]
		= Q' \subseteq B \, ,
	\]
	which proves the maximality of $Q'$ in $\kk[t, y, y^{-1}]$.
	
	\item We claim that $\pp Q_h \cap \kk[t, y, y^{-1}]$ 
	is the crucial maximal ideal of $Q'$. Clearly, 
	$\pp Q_h$ is the crucial maximal ideal of $Q_h$. 
	If $\pp Q_h \cap \kk[t, y, y^{-1}]$ would not be the 
	crucial maximal ideal of $Q'$, then
	$\Spec Q_h \to \Spec Q'$ would send $\pp Q_h$ to a point of the open
	subset $\Spec \kk[t, y, y^{-1}]$ of 
	$\Spec Q'$. This would imply
	that $\kk[t, y, y^{-1}] \subseteq Q_h$, a contradiction.
	
	\item We claim that $Q' \in \NNN_0$. By ii), $\pp Q_h \cap \kk[t, y, y^{-1}]$
	is the crucial maximal ideal of $Q'$ and it contains $t$.
	By the equality \eqref{eq23} 
	we get $y=q+tf$ for some $q \in Q$, $f \in \kk[t, y]$. 
	Since $Q \subseteq Q'$ and $\kk[t, y^{-1}] \subseteq Q'$, the homomorphism
	\[
		 Q' \longrightarrow \kk[t, y, y^{-1}] / t \kk[t, y, y^{-1}]
	\]
	is surjective. With i) we get $Q' \in \NNN_0$. 
	
	\item We claim that $Q' \cap \kk[t, y] = Q$. This follows from the fact that
	$Q \subseteq Q' \cap \kk[t, y] \subsetneq \kk[t, y]$ and from the maximality
	of $Q$ in $\kk[t, y]$. 
	\end{enumerate}
	This proves the surjectivity.
\end{proof}	      

Let us interpret the map $\NNN_0 \to \MMM_0$, 
$A \mapsto A \cap \kk[t, y]$ in geometric terms. For this we introduce the following
terminology.

\begin{definition}
	We call a dominant morphism $Y \to X$ of affine schemes
	an \emph{(extending) minimal morphism}, if $\Gamma(X, \OOO_X)$ 
	is an (extending) maximal subring of $\Gamma(Y, \OOO_Y)$.
	Moreover, the point in $X$ which corresponds to 
	the crucial maximal ideal of $\Gamma(X, \OOO_X)$ we call
	the \emph{crucial point} of $X$.
\end{definition}

Let us denote by $\pr \colon \AA^2_{\kk} \to \AA_{\kk}^1$ the projection 
$(t, y) \mapsto t$. The set $\MMM_0$ corresponds to the extending minimal
morphisms 
$\psi \colon \AA^2_{\kk} \to X$ such that 
$\pr \colon \AA^2_{\kk} \to \AA_{\kk}^1$ factorizes as
\[
	\AA^2_{\kk} \stackrel{\psi}{\longrightarrow} X \longrightarrow \AA_{\kk}^1 \, ,
\]
and such that the crucial point of $X$
is sent onto $0 \in \AA^1_{\kk}$ via $X \to \AA^1_{\kk}$.
The set $\NNN_0$ corresponds to the extending 
minimal morphisms $\varphi \colon \AA_{\kk}^1 \times \AA_{\kk}^\ast \to Y$ 
such that the open immersion $\AA_{\kk}^1 \times \AA_{\kk}^\ast \to \AA_{\kk}^2$, 
$(t, y) \mapsto (t, y^{-1})$ factorizes as 
\[
	\AA_{\kk}^1 \times \AA_{\kk}^\ast \stackrel{\varphi}{\longrightarrow} Y
	\longrightarrow \AA_{\kk}^2
\]
and such that the image of 
$\{ 0 \} \times \AA_{\kk}^\ast$ under $\varphi$ is closed in $Y$.

\begin{proposition}
	\label{prop:glueing}
	Let $\varphi \colon \AA_{\kk}^1 \times \AA_{\kk}^\ast \to Y$ 
	be an extending 
	minimal morphism corresponding to an element $A \in \NNN_0$. Then 
	\[
		\Spec AÊ\cap \kk[t, y]  = Y \cup_\varphi \AA^2_{\kk}
	\]
	where $Y \cup_\varphi \AA^2_{\kk}$ denotes the glueing via
	$\AA^2_{\kk} \stackrel{\sigma}{\longleftarrow} \AA_{\kk}^1 \times \AA_{\kk}^\ast 
	\stackrel{\varphi}{\longrightarrow} Y$ where $\sigma$
	is the open immersion defined by 
	$\sigma(t, y) = (t, y)$.
\end{proposition}	
	
\begin{proof} 
By Theorem~\ref{thm:MandN} we have the following commutative diagram
\[
	\xymatrix{	
		\AA_{\kk}^1 \times \AA_{\kk}^\ast \ar@{^(->}[d]_{\sigma} 
		\ar[rrr]^-{\varphi}_-{\textrm{extend. minimal mor.}} &&& Y \ar[d] \\
		\AA_{\kk}^1 \times \AA_{\kk}^1 \ar[rrr]_-{\textrm{extend. minimal mor.}} 
		&&& \Spec A \cap \kk[t, y] \, .
	}
\]
By Lemma~\ref{maxsub.lem},
there exists a regular function $h$ on $\Spec A \cap \kk[t, y]$
that does not vanish at the crucial point of $\Spec A \cap \kk[t, y]$ 
and we have 
\[
	A_h = (A \cap \kk[t, y])_h \, .
\] 
Thus $Y \to \Spec A \cap \kk[t, y]$
restricts to an open immersion on $Y_h$. By the commutativity of the diagram, 
it follows that $Y \to \Spec A \cap \kk[t, y]$ restricts to an open
immersion on $\varphi(\AA_{\kk}^1 \times \AA_{\kk}^\ast)$. 
By Lemma~\ref{maxsub.lem}, 
the morphism 
$Y \to \Spec A \cap \kk[t, y]$ maps the crucial point of $Y$ to that one of 
$\Spec A \cap \kk[t, y]$. Hence, $Y_h$ contains the crucial point of $Y$. 
In summary, we get
that $Y \to \Spec A \cap \kk[t, y]$ is an open immersion.
Thus all morphisms in the diagram above are open immersions.
Moreover, $\varphi$ induces an isomorphism 
$\AA_{\kk}^1 \times \AA_{\kk}^\ast \to Y \cap \AA_{\kk}^1 \times \AA_{\kk}^1$ where 
we consider $Y \cap \AA_{\kk}^1 \times \AA_{\kk}^1$ as an open subset of 
$\Spec A \cap \kk[t, y]$. Hence $\Spec A \cap \kk[t, y]$
is the claimed glueing.
\end{proof}	      
	
\section{Acknowledgement}
	We would like to thank J\'er\'emy Blanc for showing us 
	Lemma~\ref{lem:Jeremy}.
	
\section{Funding}
	The second author gratefully acknowledge support by the Swiss National
	Science Foundation (Schweizerischer National Fonds) [148627].
	
\providecommand{\bysame}{\leavevmode\hbox to3em{\hrulefill}\thinspace}
\providecommand{\MR}{\relax\ifhmode\unskip\space\fi MR }
\providecommand{\MRhref}[2]{%
  \href{http://www.ams.org/mathscinet-getitem?mr=#1}{#2}
}
\providecommand{\href}[2]{#2}

\end{document}